\documentclass[]{siamart1116}

\usepackage{lipsum}
\usepackage{amsfonts}
\usepackage{graphicx}
\usepackage{epstopdf}

\usepackage{setspace}
\usepackage{amssymb}
\usepackage{verbatim}
\usepackage{mathtools}
\usepackage{algorithm}
\usepackage{algpseudocode}

\ifpdf
\DeclareGraphicsExtensions{.eps,.pdf,.png,.jpg}
\else
\DeclareGraphicsExtensions{.eps}
\fi

\numberwithin{theorem}{section}

\newcommand{\TheTitle}{An adaptive memory method for accurate and efficient computation of the Caputo fractional derivative} 
\newcommand{\TheAuthors}{D. Yoon and D. You}

\headers{An adaptive memory method for the Caputo fractional derivative}{\TheAuthors}

\title{{\TheTitle}}
	
\author{
	Daegeun Yoon
	\and
	Donghyun You\thanks{Department of Mechanical Engineering, Pohang University of Science and Technology, 77 Cheongam-Ro, Nam-Gu, Pohang, Gyeongbuk 37673, South Korea (\email{dhyou@postech.ac.kr}).}	
}

\usepackage{amsopn}

\ifpdf
\hypersetup{
  pdftitle={\TheTitle},
  pdfauthor={\TheAuthors}
}
\fi

\begin{document}

\maketitle

\begin{abstract}
A fractional derivative is a temporally nonlocal operation which is computationally intensive due to inclusion of the accumulated contribution of function values at past times.
In order to lessen the computational load while maintaining the accuracy of the fractional derivative, a novel numerical method for the Caputo fractional derivative is proposed. 
The present adaptive memory method significantly reduces the requirement for computational memory for storing function values at past time points and also significantly improves the accuracy by calculating convolution weights to function values at past time points which can be non-uniformly distributed in time. The superior accuracy of the present method to the accuracy of the previously reported methods is identified by deriving numerical errors analytically. The sub-diffusion process of a time-fractional diffusion equation and the creeping response of a fractional viscoelastic model are simulated to demonstrate the accuracy as well as the computational efficiency of the present method.
\end{abstract}



\section{Introduction}
Recently, the fractional calculus has often been employed to model or understand complex biological phenomena such as anomalous sub-diffusion under cell division~\cite{Selhuber-Unkel2009}, wave propagation thorough tissues~\cite{Sebaa2006}, and various mechanical behaviors of individual cells~\cite{Craiem2010}.

Especially, research has been conducted to understand the nonlinear deformability of a red blood cell (RBC), which is crucial in analyzing circulatory diseases~\cite{Suresh2005}. In the previous studies, the nonlinear viscoelasticity of an RBC membrane has been modeled using classical linear elements like springs and viscous dampers~\cite{Dao2003,Mills2004,Yoon2016}.
Linear combinations of classical elements lead to solutions with exponential variations to the corresponding constitutive equations. However, as Yoon {\it et al}.~\cite{Yoon2008} revealed in optical tweezer experiments, the dynamic stiffness of an RBC membrane is found to follow a power function of time instead of the conventionally considered exponential functions. 

Yoon~{\it et al.}~\cite{Yoon2008} claimed that the breakage of cross-linked proteins on an RBC membrane, severely changes mechanical properties of an RBC under deformation. In their stretching experiment, the dynamic stiffness of an RBC membrane was decreased over time and was fitted with a power function. Similarly, Puig-de-Morales-Marinkovi~{\it et al.}~\cite{Puig2007} attached ferromagnetic beads on RBC membranes and measured the complex modulus of an RBC by using a magnetic twisting cytometry. They also found that the loss modulus is fitted with a power function of the loading rate. Both experimental studies showed power-law behaviors of mechanical properties of an RBC in the time domain and in the frequency domain, which are not adequate to be modeled with the classical linear elements. Meanwhile, Craiem and Magin~\cite{Craiem2010} showed the possibility of the fractional calculus to model power-law behaviors of an RBC in time and frequency domains mathematically.

The fractional derivative is a nonlocal operation which has a convolution integral over time. Consequently, the fractional derivative is capable of prescribing the memory effects of a physical system. However, time points for numerical evaluation of the fractional derivative are accumulated as time elapses, thereby requiring exorbitant memory usage. Moreover, the convolution integral is evaluated on every time points, which takes a significant amount of computational cost. To resolve the problem, Podlubny~\cite{Podlubny1999} developed a simple and easy-to-implement numerical method known as the fixed memory method. The method is based on the characteristic of the kernel function in the convolution integral. The kernel function decays along a backward time direction so that old time points are less weighted. 
Utilizing the characteristic, the method stores recent time points within a given memory length $T$ from the current time. The method is found to be quite successful in achieving huge reduction of memory space and operation counts. Nevertheless, the abrupt elimination of past time points becomes a noticeable source of numerical error.

As an alternative, Ford and Simpson \cite{Ford2001} proposed a delicate numerical method, called the nested mesh method which considers whole time history. Instead of the abrupt elimination of past time points, this method adaptively allocates less time points on older time history, leading to substantial reduction of computational cost while providing better accuracy than that of the fixed memory method. Lubich and others \cite{Lopez-Fernandez2008,Lubich1988,Lubich2002} proposed a fractional linear multi-step method to approximate the convolution integral with reduced memory requirement. The method shows fast convergence to an exact solution obtained through inverse Laplace transform of a kernel function using contour integrals. However, the numerical procedure is found to become too complicated to deal with other various kernel functions. 

Recently, MacDonald~{\it et al.}~\cite{MacDonald2015} proposed an adaptive memory method, which manages time points distribution in a cost effective way. The underlying principle of the method is similar to that of the nested mesh method, but it achieves significant reduction not only in computational cost but also in memory requirement. However, the method is exclusive to the Gr{\"u}nwald-Lenikov fractional derivative which is adequate only for uniformly distributed time points. Consequently, the method is incapable of accurately approximating the convolution weights for adaptively distributed time points.

In order to overcome drawbacks of the aforementioned methods for the fractional derivative, in the present study, a novel numerical method for the Caputo fractional derivative, which achieves significant reduction in memory requirement and computational cost while significantly improves numerical accuracy, is proposed. As will be discussed in the following sections, the present numerical method is much easier to implement than the fractional multi-step methods. Improvement of numerical accuracy is obtained through the exact evaluation of the discretized convolution integral over non-uniformly distributed time points. The L1-norm error function is derived to characterize error behaviors of the present adaptive memory method with various computational conditions. As practical examples, the sub-diffusion process of a time fractional diffusion equation and the creep response of a fractional Kelvin-Voigt model are simulated to demonstrate the efficacy of the present method.

The present paper is organized as follows: the definition of the fractional derivative and the L1 discretization scheme are introduced first in Section 2. In Section~\ref{sec:sec3}, details of the present method and the computational aspects are presented. The accuracy of the present method is discussed in Section~\ref{sec:sec4} along with analytic and computational analyses of the method in Sections~\ref{sec:sec5} and~\ref{sec:sec6}. Concluding remarks are followed in Section~\ref{sec:sec7}.

\section{The Caputo fractional derivative and the L1 scheme}
\label{sec:sec2}
\subsection{Definition of the Caputo fractional derivative}
\label{sec:sec2.1}
The Cauchy formulation of repeated integrations of a function $f(t)$ can be easily extended to a real-number domain. This is the basic idea of fractional integrals, which often means the Riemann-Liouville fractional integral. Based on the concept of fractional integrals, there are several different definitions for fractional derivatives while only the Caputo fractional derivative is considered in the present study due to its practicality (see Remark~\ref{rmk:rmk1} and Remark~\ref{rmk:rmk2}).
\begin{definition}
	The Riemann-Liouville fractional integral of a function $f(t)$ for a fractional order $\alpha\in\mathbb{R}_+$, is defined as follows:
	\begin{equation}
	{J^\alpha f(t)=\frac{1}{\Gamma(\alpha)}\int_0^t \! \frac{f(\tau)}{(t-\tau)^{1-\alpha}} \, \mathrm{d}\tau},
	\label{eqn:eq1}
	\end{equation}
	where $t>0$ and $\Gamma$ is a Gamma function.
	From Eq.~(\ref{eqn:eq1}), the Caputo fractional derivative of a function $f(t)$ is defined as follows:
	\begin{equation}
	{D^\alpha f(t) = J^{z-\alpha} f^{(z)}(t)=\frac{1}{\Gamma(z-\alpha)}\int_0^t \! \frac{f^{(z)}(\tau)}{(t-\tau)^{1-z+\alpha}} \, \mathrm{d}\tau,}
	\label{eqn:eq2}
	\end{equation}
	where $z-1<\alpha<z$ and $z\in \mathbb{N}$.
	\label{df:df1}
\end{definition}

\begin{remark}
	The derivative of a constant is expected to be zero in classical calculus. However, this is not always true in the fractional derivative depending on the definition. Considering Eq.~(\ref{eqn:eq2}), the Caputo fractional derivative of a constant $f(t)=c$ is always zero, meanwhile, the Riemann-Liouville fractional derivative results in $ct^{-\alpha} / \Gamma(1-\alpha)$ \cite{Li2015}.
	\label{rmk:rmk1}
\end{remark}

\begin{theorem} 
	The Laplace transformation of the Riemann-Liouville fractional integral is given by
	\begin{equation}
	{\mathcal{L} \{ J^\alpha f(t) \}=s^{-\alpha}F(s),}
	\label{eqn:eq3}
	\end{equation}
	where $t>0$ and $f(t)$ is a causal function vanishing for $t\leq0$.
	\label{thm:thm1}
\end{theorem}
\begin{proof}
	After introducing a kernel function $\Phi_\alpha(t)$ in Eq.~(\ref{eqn:eq1}), the Riemann-Liouville fractional integral is represented by a convolution integral as follows~\cite{Loverro2004}: 
	\begin{equation}
	{J^\alpha f(t)=\Phi_\alpha(t) * f(t),}
	\label{eqn:eq4}
	\end{equation}
	where $\Phi_\alpha(t)=\frac{t^{\alpha-1}}{\Gamma(\alpha)}$.
	
	Using the linear property of the Laplace convolution, the Laplace transformed Riemann-Liouville fractional integral then becomes 
	\begin{equation}
	{\mathcal{L} \{ J^\alpha f(t) \}=\mathcal{L} \{ \Phi_\alpha(t) \} \mathcal{L} \{ f(t) \}=s^{-\alpha}F(s),}
	\label{eqn:eq5}
	\end{equation}
	where $F(s)$ is a Laplace transformed function of $f(t)$.
\end{proof}

\begin{corollary} 
Through Theorem~\ref{thm:thm1}, the Laplace transformation of the Caputo fractional derivative is expressed with integer order derivatives as follows:
\begin{equation}
{\mathcal{L} \{ J^{z-\alpha} f^{(z)}(t) \}= s^{\alpha-z}\mathcal{L} \{ f^{(z)}(t) \}=s^{\alpha}F(s)-\sum_{k=0}^{z-1}s^{\alpha-k-1}f^{(k)}(0).}
\label{eqn:eq6}
\end{equation}
\label{cor:cor1}
\end{corollary}
\begin{remark} 
Integer order derivatives ($f(0), \cdots, f^{(z-1)}(0)$) are used as initial conditions of fractional differential equations~\cite{Loverro2004}. 
\label{rmk:rmk2}
\end{remark}

\subsection{L1 scheme}
\label{sec:sec2.2}
A standard discretization scheme for the Caputo fractional derivative known as the L1 scheme is as follows~\cite{Oldham2006}:
\begin{definition}
	The Caputo fractional derivative discretized with the L1 scheme for $0<\alpha<1$, is as follows: 
	\begin{equation}
	{D^\alpha f(t) = \frac{1}{\Gamma(1-\alpha)}\sum_{k=0}^{n-1}\omega_t^k \frac{f^{k+1}-f^{k}}{t^{k+1}-t^k},}
	\label{eqn:eq7}
	\end{equation}
	\begin{equation*}
		{\omega_t^k=\int_{t^k}^{t^{k+1}} \! (t^n-\tau)^{-\alpha}\mathrm{d}\tau,}
	\end{equation*}	
	where $n$ denotes the current time step and $f^k=f(t^k)$ where $t^k$ is time at the $k$-th time point.
	\label{df:df2}
\end{definition}

\begin{theorem} 
	Through a von Neumann stability analysis, the L1 scheme for a time fractional diffusion equation leads to an unconditionally stable algorithm \cite{Yuste2012} as follows:
	\begin{equation}
	{|\delta^n| \leq |\delta^k|_{max}, \ }
	\label{eqn:eq8}
	\end{equation}
	where $|\delta^n|$ is the absolute value of $f(t^n)$ and $|\delta^k|_{max}$ is the maximum absolute value of $f(t^k)$ for $0\leqslant k < n$.
	\label{thm:thm2}
\end{theorem}

\section{Computational methods}
\label{sec:sec3}
\subsection{The previous fixed and adaptive memory methods}
\label{ssec:subsec3_1}
\begin{figure}
	\centering
	\includegraphics[width=0.6\linewidth]{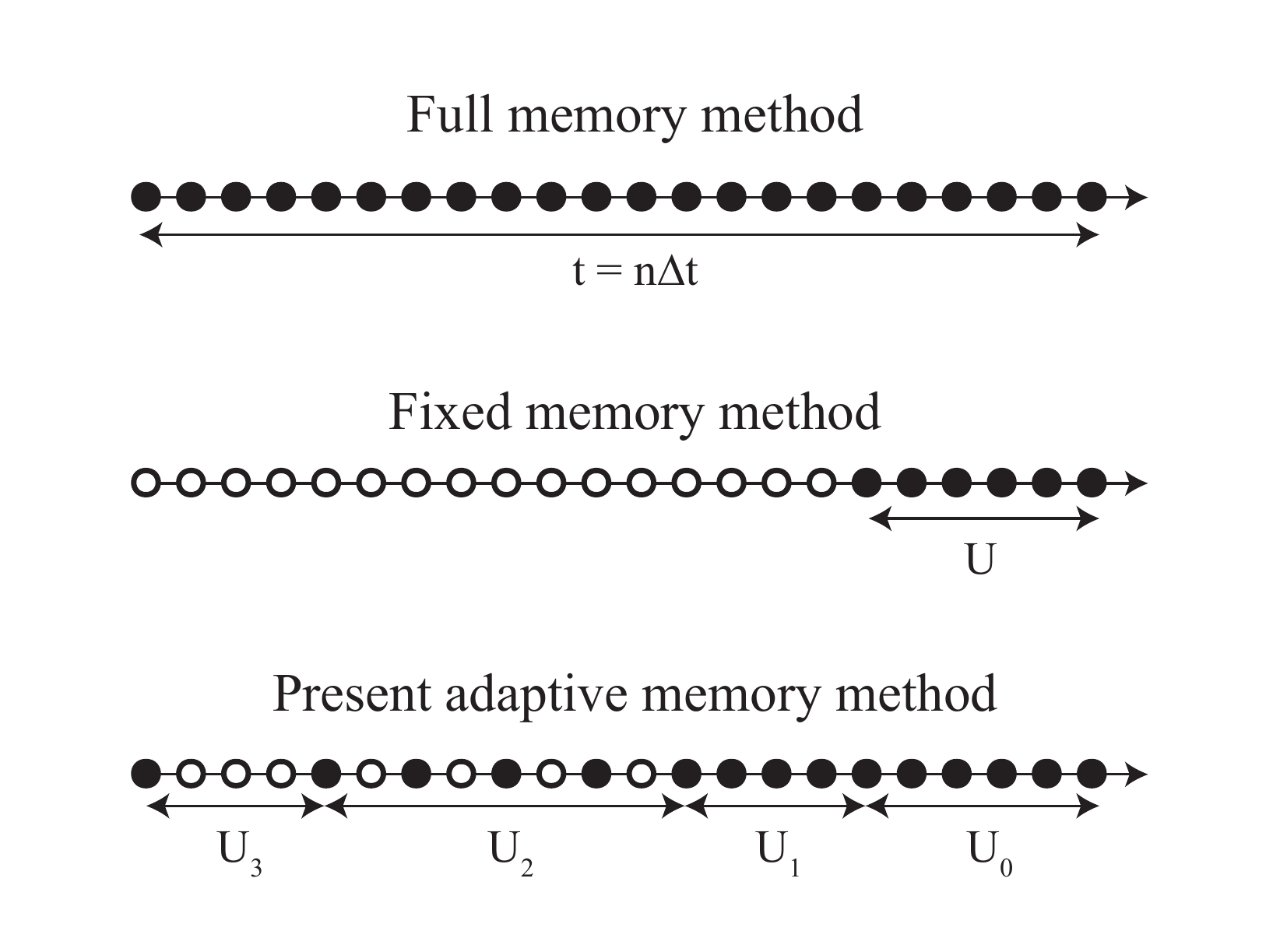}
	\caption{Distribution of time points of the full memory method (top), the fixed memory method (middle), and the present adaptive memory method (below). $\bullet$ and $\circ$ indicate stored time points and eliminated time points, respectively.}
	\label{fig:figure1}
\end{figure}
To calculate the Caputo fractional derivative, the previous time points are required in the convolution integral. This makes the fractional derivative as a nonlocal operation while requiring significant amounts of computational time and memory.
To address the problem, Podlubny \cite{Podlubny1999} proposed so called a fixed memory method. This is one of the simplest methods limiting the increase of computational cost and the memory requirement by setting time points in [0,t-T) are vanished, where $T$ is a memory length. The underlying principle of the fixed memory method is that function values relatively closer to the current time have greater contribution to the solution than function values at time points away from the current time~\cite{MacDonald2015}. For a given fixed memory length $T$, the fixed memory method is explained in Algorithm~\ref{algo:algo1} and in Fig.~\ref{fig:figure1}.
\begin{algorithm}
	\caption{Fixed memory method}
	\label{algo:algo1}
	\begin{algorithmic}
		\State Let $U$ have a subset $U_0$ which stores time points in $[t-T,t]$
		\State $U_0=\{\{f^{p}_0,t^{p}_0\},\cdots,\{f^{k}_0,t^{k}_0\},\cdots,\{f^{0}_0,t^{0}_0\}\}$
		\State Add a new time point at the current step $n$ to the subset $U_0$
		\State $U_0=U_0+\{f^{p+1}_0=f^n,t^{p+1}_0=t^n\}$
		\If{Time interval of $U_0 > T$}
		\State Remove the oldest time point $\{f^{0}_0,t^{0}_0\}$ 
		\EndIf
		\State Evaluate the fractional derivative according to the time sequence of stored time points in $U$ through Eq.~(\ref{eqn:eq9}).
	\end{algorithmic}
\end{algorithm}
\begin{equation}
{D^\alpha f(t) = \frac{1}{\Gamma(1-\alpha)}\sum_{\{f^k,t^k\}\in U}^{}\omega_t^k \frac{f^{k+1}-f^{k}}{t^{k+1}-t^{k}},}
\label{eqn:eq9}
\end{equation}
\begin{equation*}
	{\omega_t^k=\int_{t^k}^{t^{k+1}} \! (t^n-\tau)^{-\alpha}\mathrm{d}\tau.}
\end{equation*} 
The memory size required for the fixed memory method is fixed during the computation, which makes the method computationally attractive. 

However, the fixed memory method becomes inaccurate when function values vary rapidly in the truncated time interval $[0,t-T)$.
As an alternative to the fixed memory method but with additional computational cost and memory, Ford and Simpson~\cite{Ford2001} proposed a nested mesh method that gives a good approximation to an exact solution.

MacDonald {\it et al.}~\cite{MacDonald2015} proposed an adaptive memory method which allocates time-points using a power law through a linked-list algorithm. In the adaptive memory method, old time points which are not utilized in the current step integration are eliminated, thereby leading to reduction of total memory allocated for computation at $n$ time step from $\mathcal{O}(n)$ of the full memory method to $\mathcal{O}(\log_2n)$~\cite{MacDonald2015}. A notable drawback of the adaptive method of MacDonald {\it et al.}~\cite{MacDonald2015} is that the method is exclusive to the Gr{\"u}nwald-Lenikov fractional derivative which is only applicable to uniformly distributed time points. Consequently, the convolution integral over non-uniformly distributed time points is inaccurately weighted.

\subsection{The proposed new adaptive memory method}
\label{ssec:subsec3_2}
To overcome the drawbacks of the previous fixed and adaptive memory methods, a new adaptive memory method, which is based on the Caputo fractional derivative, is proposed in the present study. Unlike in the Gr{\"u}nwald-Lenikov fractional derivative, it is not necessary to be confined to uniformly distributed time points in the Caputo fractional derivative since the derivative contains a continuous convolution integral (see Definition~\ref{df:df2}). 
Therefore, in the present method, it is capable of calculating exact convolution weights even for time-points distributed using a power law, thereby leading to significant improvement in accuracy (details are discussed in Sections~\ref{sec:sec5} and~\ref{sec:sec6}). 
Some of old time points are eliminated based on a power-law distribution, and the rest of time points are grouped into subsets $U_l$ (see Fig.~\ref{fig:figure1} for a graphical illustration). For a given memory length $T$, step-by-step details of the present algorithm can be summarized as in Algorithm~\ref{algo:algo2}.
\begin{algorithm}
	\caption{The present new adaptive memory method}
	\label{algo:algo2}
	\begin{algorithmic}[0]
		\State Let $U$ have subsets $U_l$ which store time points as follows:
		\State $U=\{ U_0,\cdots,U_l,\cdots,U_L \}$, $U_l=\{\{f^p_l,t^p_l\},\cdots,\{f^0_l,t^0_l\}\}$
		\State Add a new time point at the current step $n$ to the subset $U_0$
		\State $U_0=U_0+\{f^{p+1}_0=f^n,t^{p+1}_0=t^n\}$
		\If{Time interval of $U_0 > T$}
		\State Move the oldest time point $\{f^{0}_0,t^{0}_0\}$ to $U_1$		
		\For{$l=1$ to $L$}
		\If{Time interval of $U_l>2^{l-1}T$}
		\State Eliminate the second oldest time point $\{f^{1}_l,t^{1}_l\}$ 
		\State Move the oldest time point $\{f^{0}_l,t^{0}_l\}$ to $U_{l+1}$
		\Else
		\State Exit Loop
		\EndIf
		\EndFor
		\EndIf
		\State Evaluate the fractional derivative according to the time sequence of stored time points in $U$ through Eq.~(\ref{eqn:eq9}).
	\end{algorithmic}
\end{algorithm}

The convolution summation in Eq.~(\ref{eqn:eq7}) needs to be evaluated from the initial step to the current step $n$. Supposed that the memory length $T=m\Delta t$, where $\Delta t$ is a time step size and $m\in\mathbb{N}$. In the full memory method, operation counts ($OC$) for the convolution summation from $t=0$ to $2^LT$, where $L\in\mathbb{N}$ is given, become ${OC}_{full\ memory\ method}=2^Lm(2^Lm+1)/2$ or approximately $\mathcal{O}((2^{L})^2)$. 
The fixed memory method results in considerable reduction of the operation counts to $OC_{fixed\ memory\ method}=m(m+1)/2+m^2(2^L-1)$ or $\mathcal{O}(2^L)$. 

Operation counts for the present adaptive memory method are approximated as follows:
\begin{equation}
{OC_{the\ present\ adaptive\ method}\approx m(m+1)/2+\sum_{l=1}^{L}2^{l-2}((2l+1)m^2+m)}.
\label{eqn:eq10}
\end{equation}

\begin{remark} 
	For a given $L$, operation counts for the present adaptive memory method are approximately $\mathcal{O}(L2^{L})$, which is similar to that of the nested mesh method by Ford and Simpson \cite{Ford2001}. $OC_{the\ present\ adaptive\ method}$ are determined as follows: in the time interval of $[0,T]$ in $U_0$, there are $m(m+1)/2$ convolution summations. In the interval of $(2^{l-1}T,2^{l}T]$ in $U_l$, there are $lm+1$ to $(l+1)m$ convolution summations with $2^{l-1}$ increments to account for eliminated time points. Therefore, operation counts for convolution summations become $2^{l-1}\{(lm+1)+(l+1)m\}m/2$, approximately. Eq.~(\ref{eqn:eq10}) is derived by summing these terms.
	\label{rmk:rmk3}
\end{remark}

\section{Accuracy}
\label{sec:sec4}
The L1-norm error function of the fixed memory method is expressed as follows \cite{Ford2001}:
\begin{equation}
{ Error= \left\vert \frac{1}{\Gamma(1-\alpha)}\int_{0}^{t-T}\frac{f'(\tau)}{(t-\tau)^\alpha}\mathrm{d}\tau \right\vert  \leqslant\frac{M}{\Gamma(2-\alpha)}(t^{1-\alpha}-T^{1-\alpha}),}
\label{eqn:eq11}
\end{equation}
where $M$ is the maximum absolute value of $f'(\tau)$ for $\tau\in[0,t]$. By increasing the memory length $T$, $Error$ becomes smaller, but at the same time, it loses benefits in operation counts and memory. According to Langlands and Henry~\cite{Langlands2005}, the order of accuracy of the L1 scheme is $\mathcal{O}(\Delta t^{2-\alpha})$. 
However, one can easily notice that the order of accuracy is degenerated to $\mathcal{O}(\Delta t^0)$ for the fixed memory method. Truncation of the past time points in the convolution integral in the fixed memory method severely degrades numerical accuracy.

In this section, the L1-norm error for the present adaptive memory method with the L1 scheme, is derived for the Caputo fractional derivative as conducted by Langlands and Henry~\cite{Langlands2005}. 
It is assumed that a function $f(t)$ can be expanded in a Taylor series around $t=0$ with an integral remainder term as follows:
\begin{equation}
{ f(t)=f(0)+tf^{'}(0)+\int_{0}^{t}\!(t-\tau)f^{''}(\tau) \ \mathrm{d}\tau. }
\label{eqn:eq12}
\end{equation}
Applying the fractional derivative operator with $0<\alpha<1$ to Eq.~(\ref{eqn:eq12}) leads to
\begin{equation}
{ \frac{\mathrm{d}^{\alpha} f(t)}{\mathrm{d}t^{\alpha}}=f(0)\frac{\mathrm{d}^{\alpha} 1}{\mathrm{d}t^{\alpha}}+f^{'}(0)\frac{\mathrm{d}^{\alpha} t}{\mathrm{d}t^{\alpha}}+\frac{\mathrm{d}^{\alpha}}{\mathrm{d}t^{\alpha}}\int_{0}^{t}\!(t-\tau)f^{''}(\tau) \ \mathrm{d}\tau }.
\label{eqn:eq13}
\end{equation}
By Definition~\ref{df:df1}, the exact expression of the fractional derivative becomes as follows: 
\begin{equation}
\frac{\mathrm{d}^{\alpha} f}{\mathrm{d}t^{\alpha}}=f^{'}(0)\frac{t^{1-\alpha}}{\Gamma(2-\alpha)}+\frac{1}{\Gamma(2-\alpha)}\int_{0}^{t}\!(t-\tau)^{1-\alpha}f^{''}(\tau) \ \mathrm{d}\tau,
\label{eqn:eq14}
\end{equation}
where the first term in Eq.~(\ref{eqn:eq13}) becomes zero as noted in Remark~\ref{rmk:rmk1}.

The accuracy of the L1 scheme is evaluated by comparing the above result with the result obtained by applying the L1 scheme to Eq.~(\ref{eqn:eq13}). From Definition~\ref{df:df2}, the fractional derivatives in Eq.~(\ref{eqn:eq13}) can be derived with the L1 scheme. The first term in Eq.~(\ref{eqn:eq13}) is easily determined to be zero. The expression of the L1 scheme operating on the second term is as follows:
\begin{equation}
{f^{'}(0)\left.{\frac{\mathrm{d}^{\alpha} t}{\mathrm{d}t^{\alpha}}}\right\vert_{L1}=\frac{f^{'}(0)}{\Gamma(1-\alpha)}\sum_{k=0}^{n-1}\omega_t^k}.
\label{eqn:eq15}	
\end{equation}
Using the below identity
\begin{equation}
{\sum_{k=0}^{n-1}\omega_t^k=\int_{0}^{t^n}(t^n-\tau)^{-\alpha}\mathrm{d}\tau=\frac{(t^n)^{1-\alpha}}{1-\alpha}},
\label{eqn:eq17}	
\end{equation}
Eq.~(\ref{eqn:eq15}) is simplified to
\begin{equation}
{f^{'}(0)\left.{\frac{\mathrm{d}^{\alpha} t}{\mathrm{d}t^{\alpha}}}\right\vert_{L1}=f^{'}(0)\frac{(t^n)^{1-\alpha}}{\Gamma(2-\alpha)}}.
\label{eqn:eq16}	
\end{equation}

Note that both numerical results of the first term and the second term are identical to the exact results of the fractional derivatives. Thus, any errors must arise from errors in the third term in Eq.~(\ref{eqn:eq13}). The third term with the L1 scheme is written as follows:
\begin{equation}
{\left.\frac{\mathrm{d}^{\alpha}}{\mathrm{d}t^{\alpha}}\int_{0}^{t} (t-\tau) \!f^{''}(\tau) \ \mathrm{d}\tau \right\vert_{L1}=\frac{1}{\Gamma(1-\alpha)}\sum_{k=0}^{n-1}\int_{t^k}^{t^{k+1}}\!(t^n-\tau)^{-\alpha}\ \mathrm{d}\tau \frac{u^{k+1}-u^{k}}{t^{k+1}-t^k}},
\label{eqn:eq18}	
\end{equation}
where
\begin{equation}
{u^k = \int_{0}^{t^k}\!(t^k-\tau)f^{''}(\tau) \ \mathrm{d}\tau }.
\label{eqn:eq19}	
\end{equation}
Evaluating the integration simplifies Eq.~(\ref{eqn:eq18}) to as follows:
\begin{equation}
{\frac{1}{\Gamma(2-\alpha)}\sum_{k=0}^{n-1}\{(t^n-t^k)^{1-\alpha}-(t^n-t^{k+1})^{1-\alpha}\} \frac{u^{k+1}-u^{k}}{t^{k+1}-t^k}}.
\label{eqn:eq20}	
\end{equation}
Note that $u^{k+1}-u^k$ has a form as follows:
\begin{equation}
\begin{split}
{u^{k+1}-u^k=\int_{0}^{t^{k+1}}\!(t^{k+1}-\tau)f^{''}(\tau) \ \mathrm{d}\tau-\int_{0}^{t^k}\!(t^k-\tau)f^{''}(\tau) \ \mathrm{d}\tau} \\
{=\int_{0}^{t^{k}}\!(t^{k+1}-t^k)f^{''}(\tau) \ \mathrm{d}\tau+\int_{t^k}^{t^{k+1}}\!(t^{k+1}-\tau)f^{''}(\tau) \ \mathrm{d}\tau}.
\label{eqn:eq21}
\end{split}	
\end{equation}
Substituting Eq.~(\ref{eqn:eq21}) into Eq.~(\ref{eqn:eq20}) leads to
\begin{equation}
\begin{split}
{\frac{1}{\Gamma(2-\alpha)}\sum_{k=0}^{n-1}\frac{(t^n-t^k)^{1-\alpha}-(t^n-t^{k+1})^{1-\alpha}}{{t^{k+1}-t^k}} \int_{t^k}^{t^{k+1}}\!(t^{k+1}-\tau)f^{''}(\tau) \ \mathrm{d}\tau}\\
{+\frac{1}{\Gamma(2-\alpha)}\sum_{k=0}^{n-1}\{(t^n-t^k)^{1-\alpha}-(t^n-t^{k+1})^{1-\alpha}\} \int_{0}^{t^{k}}\!f^{''}(\tau) \ \mathrm{d}\tau}.
\label{eqn:eq22}	
\end{split}
\end{equation}
The second term in Eq.~(\ref{eqn:eq22}) is zero for $k=0$ and is rewritten as follows:
\begin{equation}
{\frac{1}{\Gamma(2-\alpha)}\sum_{k=1}^{n-1}\{(t^n-t^k)^{1-\alpha}-(t^n-t^{k+1})^{1-\alpha}\} \sum_{q=0}^{k-1}\int_{t^q}^{t^{q+1}}\!f^{''}(\tau) \ \mathrm{d}\tau}.
\label{eqn:eq23}	
\end{equation}
The order of summations $k$ and $q$ in Eq.~(\ref{eqn:eq23}) is switched as
\begin{equation}
{\frac{1}{\Gamma(2-\alpha)}\sum_{q=0}^{n-2}\int_{t^q}^{t^{q+1}}\!f^{''}(\tau) \ \mathrm{d}\tau \sum_{k=1+q}^{n-1}\{(t^n-t^k)^{1-\alpha}-(t^n-t^{k+1})^{1-\alpha}\} }.
\label{eqn:eq24}	
\end{equation}
The above equation is simplified to
\begin{equation}
{\frac{1}{\Gamma(2-\alpha)}\sum_{k=0}^{n-1}\int_{t^k}^{t^{k+1}}\!f^{''}(\tau)  (t^n-t^{k+1})^{1-\alpha} \ \mathrm{d}\tau}.
\label{eqn:eq25}	
\end{equation}
Substituting Eq.~(\ref{eqn:eq25}) into the second term of Eq.~(\ref{eqn:eq22}) leads to
\begin{equation}
\begin{split}
\frac{1}{\Gamma(2-\alpha)}\sum_{k=0}^{n-1}\int_{t^k}^{t^{k+1}}\!f^{''}(\tau)  \frac{t^{k+1}-\tau}{t^{k+1}-t^k} \{  (t^n-t^k)^{1-\alpha}-(t^n-t^{k+1})^{1-\alpha} \} \\
+f^{''}(\tau)(t^n-t^{k+1})^{1-\alpha} \mathrm{d}\tau.
\end{split}
\label{eqn:eq26}	
\end{equation}
The absolute value obtained by subtracting the exact results from the numerical results, is defined as the L1-norm error of the present adaptive memory method. Thus, the error function is finally given as follows:
\begin{equation}
\begin{split}
&\left| \frac{\mathrm{d}^{\alpha} f}{\mathrm{d}t^{\alpha}}-\left.{\frac{\mathrm{d}^{\alpha} f}{\mathrm{d}t^{\alpha}}}\right\vert_{L1}  \right|=
\left|  \frac{1}{\Gamma(2-\alpha)}\sum_{k=0}^{n-1}\int_{t^k}^{t^{k+1}}\! f^{''}(\tau) (t^n-\tau)^{1-\alpha} \right.\\ 
&\left. -f^{''}(\tau)\frac{t^{k+1}-\tau}{t^{k+1}-t^k}  \{(t^n-t^k)^{1-\alpha}-(t^n-t^{k+1})^{1-\alpha} \} -f^{''}(\tau)(t^n-t^{k+1})^{1-\alpha} \mathrm{d}\tau \right|.
\end{split}	
\label{eqn:eq27}
\end{equation}
By denoting the maximum absolute value of the second derivative as
\begin{equation}
M=\mathrm{max} |f^{''}(\tau)|, \ \tau\in[0,t^n],
\label{eqn:eq28}
\end{equation}
the L1-norm error function becomes
\begin{equation}
\begin{split}
\left| \frac{\mathrm{d}^{\alpha} f}{\mathrm{d}t^{\alpha}}-\left.{\frac{\mathrm{d}^{\alpha} f}{\mathrm{d}t^{\alpha}}}\right\vert_{L1}  \right| \leqslant
\frac{M }{\Gamma(2-\alpha)}\sum_{k=0}^{n-1}\left| \int_{t^k}^{t^{k+1}}\! (t^n-\tau)^{1-\alpha} \right.\\  
\left. -\frac{t^{k+1}-\tau}{t^{k+1}-t^k} \{(t^n-t^k)^{1-\alpha}-(t^n-t^{k+1})^{1-\alpha} \} -(t^n-t^{k+1})^{1-\alpha} \mathrm{d}\tau \right|.
\end{split}	
\label{eqn:eq29}
\end{equation}

\begin{theorem} The integral term in Eq.~(\ref{eqn:eq29}) is always positive for $0<\alpha<1$.
	\begin{equation}
	\begin{split}
	& \int_{t^k}^{t^{k+1}}\! (t^n-\tau)^{1-\alpha}  -\frac{t^{k+1}-\tau}{t^{k+1}-t^k} \{(t^n-t^k)^{1-\alpha}-(t^n-t^{k+1})^{1-\alpha} \} \\
	& -(t^n-t^{k+1})^{1-\alpha} \mathrm{d}\tau >0
	\end{split}	
	\label{eqn:eq31}
	\end{equation}
	\label{thm:thm3}
\end{theorem}
\begin{proof}
	Eq.~(\ref{eqn:eq31}) can be recast with a fractional power function $f(\tau)=(t^n-\tau)^{1-\alpha}$ as follows:
	\begin{equation*}
		\int_{t^k}^{t^{k+1}}\! f(\tau) \mathrm{d}\tau  -\frac{1}{2} \{f(t^k)+f(t^{k+1})\}(t^{k+1}-t^{k})>0.
	\end{equation*}
	The left side of the above inequality is the error of the trapezoidal rule obtained from the integration of the fractional power function $f(\tau)=(t^n-\tau)^{1-\alpha}$ over $[t^{k},t^{k+1}]$. The inequality is satisfied if and only if when the fractional power function is concave on its domain $t^{k}<\tau<t^{k+1}$.
	The second derivative of the fractional power function 
	\begin{equation*}
		f''(\tau)=-\alpha(1-\alpha)(t^n-\tau)^{-\alpha-1}
	\end{equation*}
	is always negative on its domain $t^{k}<\tau<t^{k+1}$, which means the function is concave and satisfies
	\begin{equation*}
		f((1-\beta)t^k+\beta t^{k+1})>(1-\beta)f(t^k)+\beta f(t^{k+1}),~0<\beta<1.
	\end{equation*}
	Therefore, the integral term in Eq.~(\ref{eqn:eq29}) is always positive for $0<\alpha<1$.
\end{proof}

According to Theorem 3, the integral term in Eq.~(\ref{eqn:eq29}) is evaluated as follows:
\begin{equation}
\begin{split}
\left| \frac{\mathrm{d}^{\alpha} f}{\mathrm{d}t^{\alpha}}-\left.{\frac{\mathrm{d}^{\alpha} f}{\mathrm{d}t^{\alpha}}}\right\vert_{L1}  \right| \leqslant \frac{M }{2\Gamma(3-\alpha)}\sum_{k=0}^{n-1} (t^n-t^k)^{1-\alpha}\{2(t^n-t^{k+1})+(t^{k+1}-t^k)\alpha\} \\- (t^n-t^{k+1})^{1-\alpha}\{2(t^n-t^k)-(t^{k+1}-t^k)\alpha \}.
\end{split}	
\label{eqn:eq30}
\end{equation}

Now the present adaptive memory method is applied to Eq.~(\ref{eqn:eq30}). Suppose that the subset $U_0$ stores $m+1$ time points and the subset $U_l$ stores $m$ time points. $T=m\Delta t$ is given with an uniform time step size of $\Delta t$.
\begin{equation}
\begin{split}
& t^n=T+(T+2T+4T+\cdots+2^{L-1}T),\\
& t^{k}_0=(T+2T+\cdots+2^{L-1}T)+k\Delta t, \\
& t^{k}_l=(2^{l}T+\cdots+2^{L-1}T)+k2^{l-1}\Delta t,
\end{split}	
\label{eqn:eq32}
\end{equation}
where $t^n$, $t^{k}_0$ and $t^{k}_l$ are current time, time of the $k$-th time point in $U_0$, and time of the $k$-th time point in $U_l$, respectively. Then, Eq.~(\ref{eqn:eq30}) is rewritten as summations of each subset as follows:
\begin{equation}
\begin{split}
&\left| \frac{\mathrm{d}^{\alpha} f}{\mathrm{d}t^{\alpha}}-\left.{\frac{\mathrm{d}^{\alpha} f}{\mathrm{d}t^{\alpha}}}\right\vert_{L1}  \right| \leqslant \frac{M }{2\Gamma(3-\alpha)} \sum_{k=0}^{m-1} [(T-k\Delta t)^{1-\alpha}\{2(T-k\Delta t)+\Delta t\alpha\}\\
&-\{T-(k+1)\Delta t\}^{1-\alpha}\{2(T-k\Delta t)-\Delta t\alpha \}]\\
&+\frac{M }{2\Gamma(3-\alpha)}\sum_{l=1}^{L}\sum_{k=0}^{m-1} [(2^lT-k2^{l-1}\Delta t)^{1-\alpha}\{2(2^lT-k2^{l-1}\Delta t)+2^{l-1}\Delta t\alpha\}\\
&-\{2^lT-(k+1)2^{l-1}\Delta t\}^{1-\alpha}\{2(2^lT-k2^{l-1}\Delta t)-2^{l-1}\Delta t\alpha \}],
\end{split}	
\label{eqn:eq33}
\end{equation}
where the first summation is for the error of subset $U_0$ and the second summation is for the error of subset $U_l$ where $l=1$ to $L$. By introducing $A(m,\alpha)$ and $B(m,\alpha)$, Eq.~(\ref{eqn:eq33}) is simplified as follows:
\begin{equation}
\left| \frac{\mathrm{d}^{\alpha} f}{\mathrm{d}t^{\alpha}}-\left.{\frac{\mathrm{d}^{\alpha} f}{\mathrm{d}t^{\alpha}}}\right\vert_{L1}  \right| \leqslant \frac{M}{2\Gamma(3-\alpha)} \left\lbrace \Delta t^{2-\alpha}A(m,\alpha)+\sum_{l=1}^{L}(2^{l-1}\Delta t)^{2-\alpha}B(m,\alpha)  \right\rbrace,
\label{eqn:eq34}
\end{equation}
where
\begin{equation*}
	\begin{split}
		& A(m,\alpha)= \\
		& \sum_{k=0}^{m-1}\left[ (m-k)^{1-\alpha}\left\lbrace 2(m-k-1)+\alpha\right\rbrace -(m-k-1)^{1-\alpha}\left\lbrace 2(m-k)-\alpha\right\rbrace \right],
		\\
		& B(m,\alpha)=\\
		&\sum_{k=0}^{m-1}\left[ (2m-k)^{1-\alpha}\left\lbrace 2(2m-k-1)+\alpha\right\rbrace -(2m-k-1)^{1-\alpha}\left\lbrace 2(2m-k)-\alpha\right\rbrace \right].
	\end{split}	
\end{equation*}
It immediately follows from Eq.~(\ref{eqn:eq34}) that $A(m,0)=0$ and $B(m,0)=0$ so that the L1 scheme is an identity operator. Also, $A(m,1)=1$ and $B(m,1)=0$ so that the scheme recovers a first derivative with $\mathcal{O}(\Delta t)$. 

\begin{figure}
	\centering
	\includegraphics[width=0.45\linewidth]{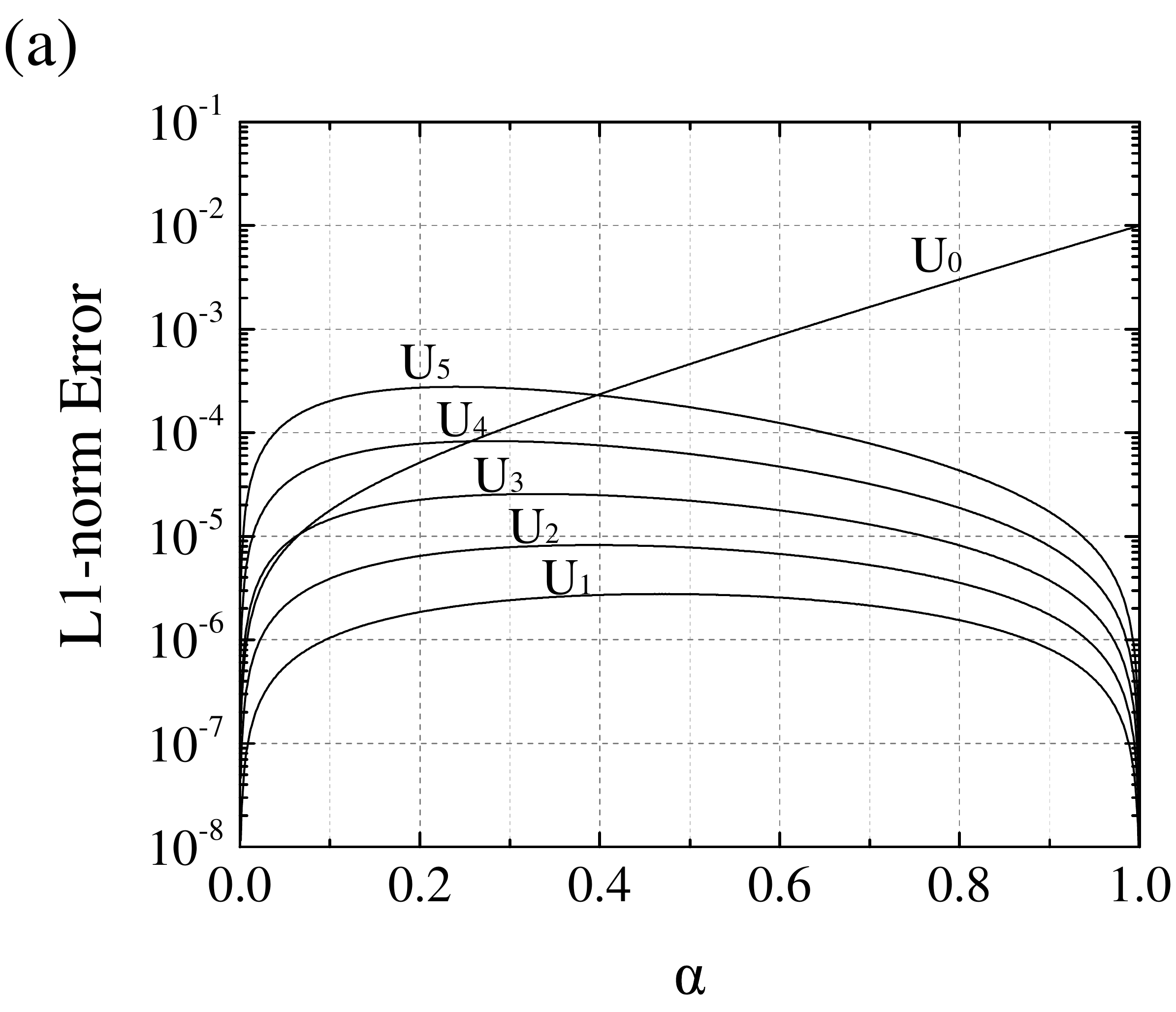}
	\includegraphics[width=0.45\linewidth]{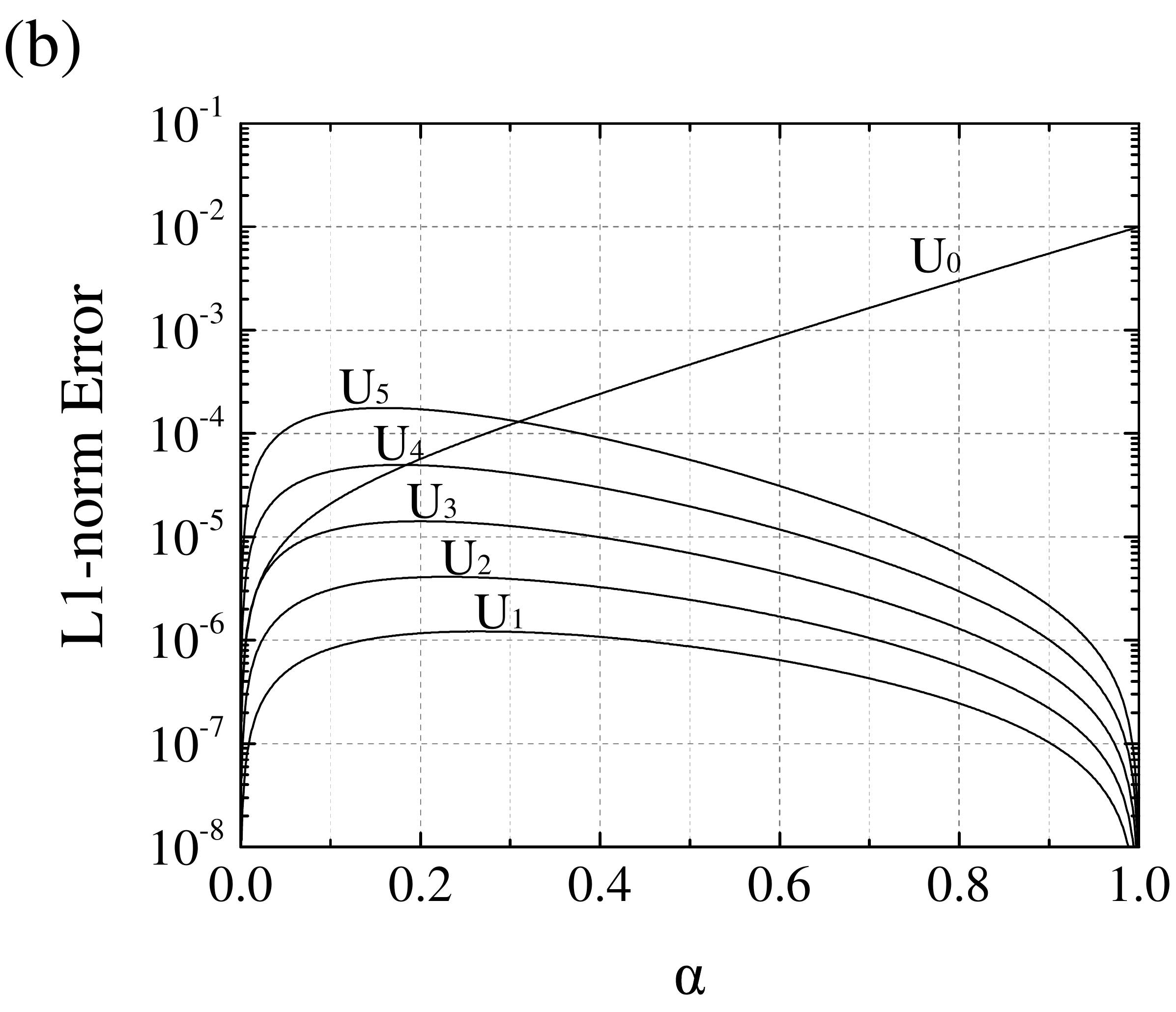}	
	\caption{L1-norm error of subset $U_l$ as a function of $\alpha$ for (a) $T=1$ and (b) $T=10$.}
	\label{fig:figure2}
\end{figure}
\begin{figure}
	\centering
	\includegraphics[width=0.45\linewidth]{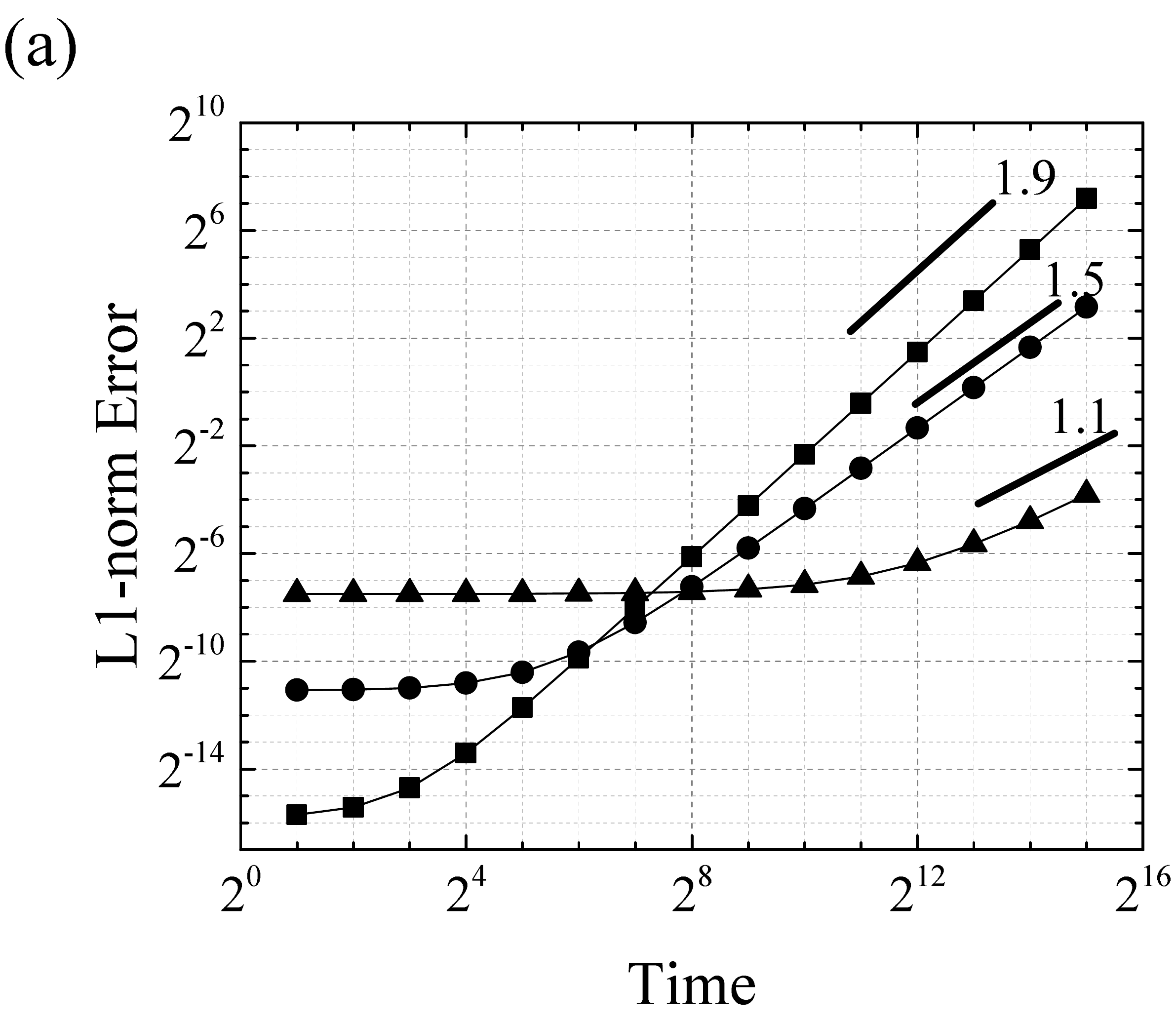}
	\includegraphics[width=0.45\linewidth]{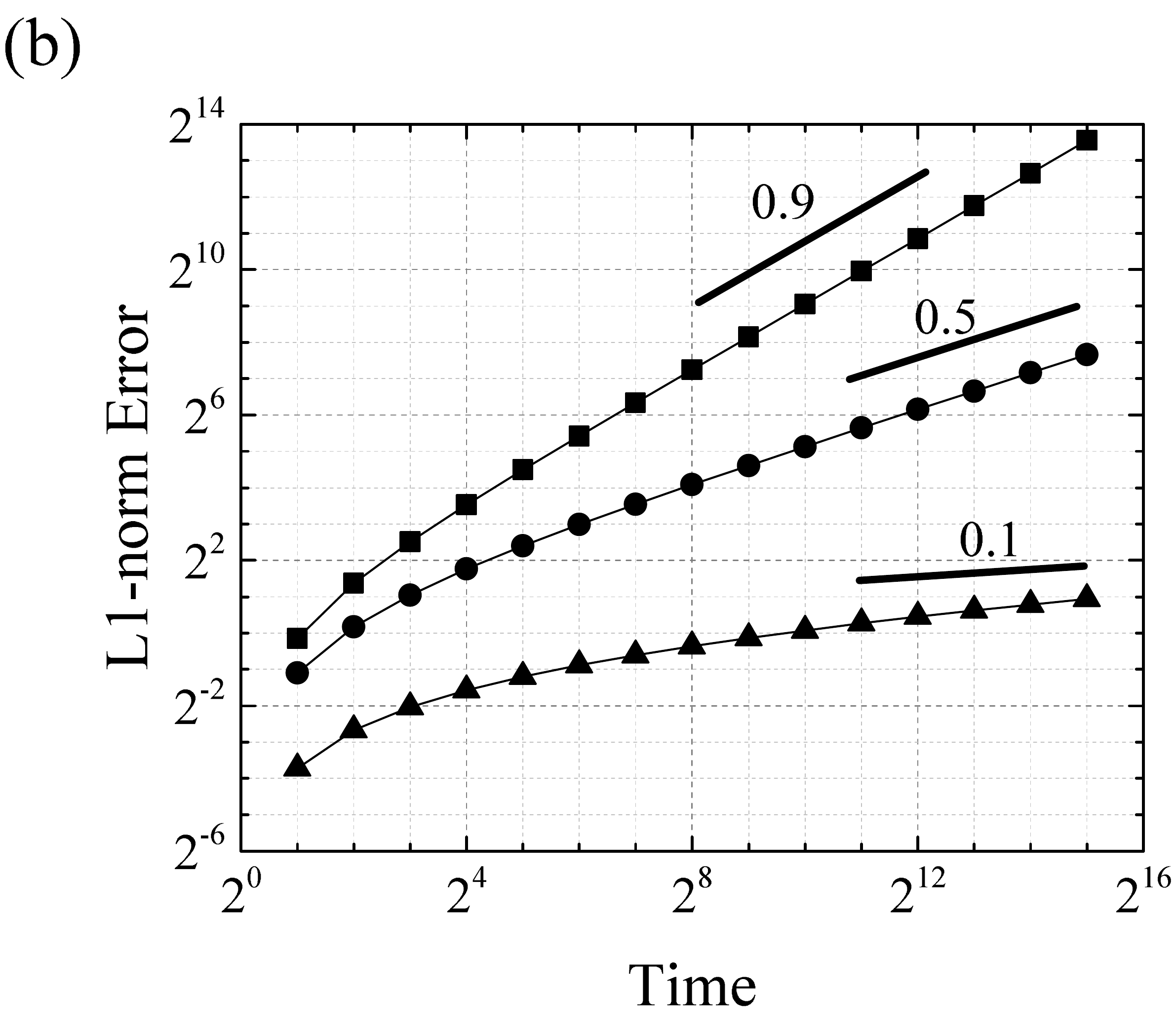}	
	\caption{L1-norm errors of computational results as a function of time for different $\alpha$ obtained using (a) the present adaptive memory method and (b) the fixed memory method. {\small{$\blacksquare$}}, $\alpha=0.1$; {\large{$\bullet$}}, $\alpha=0.5$; {\small{$\blacktriangle$}}, $\alpha=0.9$.}
	\label{fig:figure3}
\end{figure}
To investigate further, the error behavior of the present adaptive memory method is examined for a function $f(t)=t^2$ as a simple test of the error estimates because its second derivative is a constant, which makes clear the contribution to error by the present method itself not by the function values $f^{''}(\tau)$ (see Eq.~(\ref{eqn:eq27})). For the same reason, a function $f(t)=t$ is taken for the fixed memory method (see Eq.~(\ref{eqn:eq11})). The time step size $\Delta t$ is $0.01$. Fig.~\ref{fig:figure2} shows the L1-norm error of each subset according to $\alpha$ for $T=1$, and $10$. The error of subset $U_0$ gradually increases when $\alpha$ approaches to $1$. Whereas the error of subset $U_l$ increases after $\alpha=0$ and then decreases to zero when $\alpha$ approaches to $1$, which means the fractional derivative becomes a local operator when $\alpha=1$. Also, the error of subset $U_l$ increases to $2^{2-\alpha}$ times along the subset index $l$ according to Eq.~(\ref{eqn:eq34}). Therefore, in early times, the error of subset $U_0$ dominates the total error, but in later the error of subset $U_L$ dominates the total error. Fig.~\ref{fig:figure3}(a) shows the L1-norm error of the present method as a function of time for $T=1$. The error gradually increases to $2-\alpha$ order in terms of time. Fig.~\ref{fig:figure3}(b) shows the L1-norm error as a function of time for the fixed memory method with $T=1$. The error increases to $1-\alpha$ order in terms of time (see Eq.~(\ref{eqn:eq11})).

\begin{figure}
	\centering
	\includegraphics[width=0.45\linewidth]{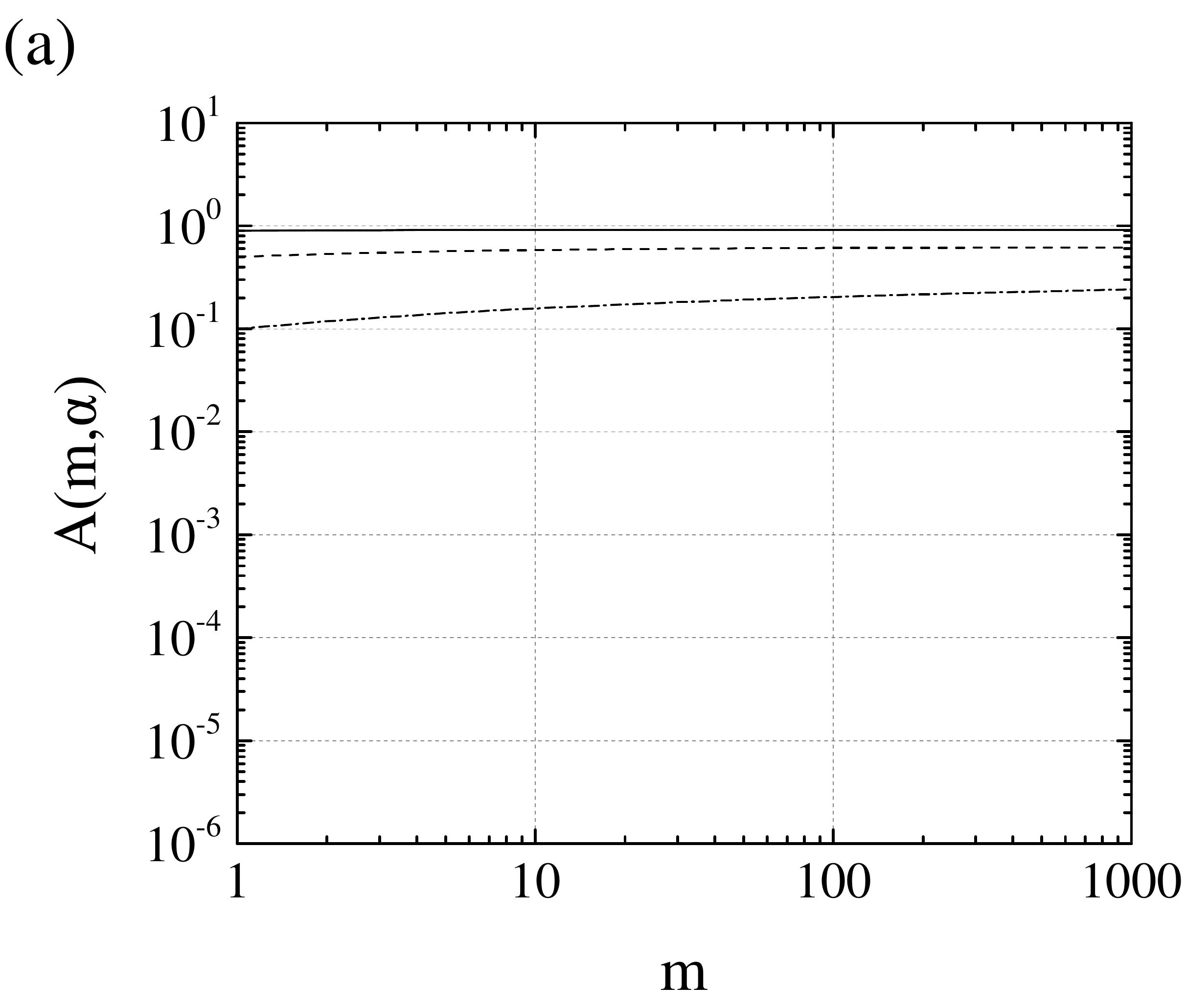}
	\includegraphics[width=0.45\linewidth]{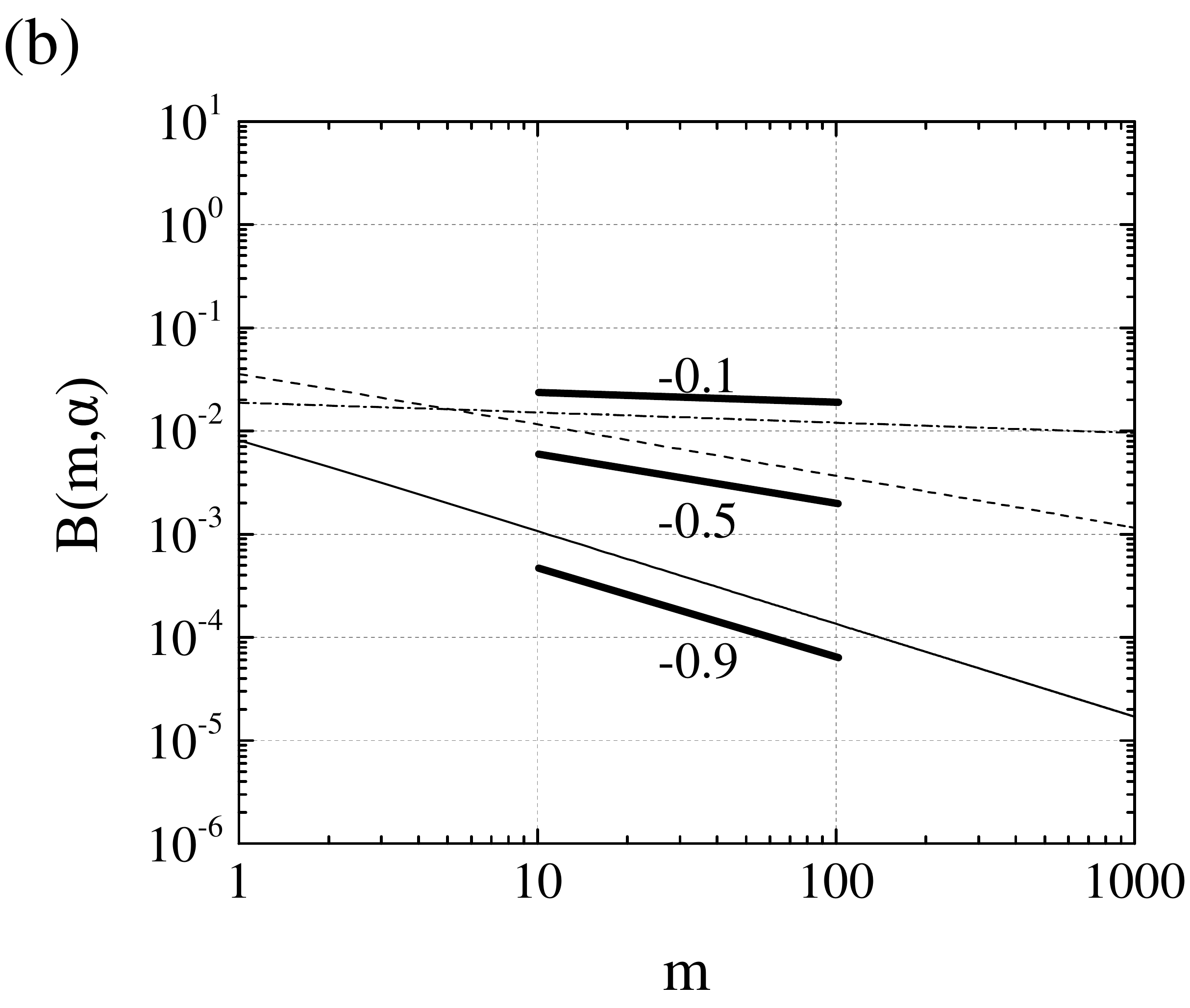}
	\caption{(a) $A(m,\alpha)$ and (b) $B(m,\alpha)$ as a function of $m$. {--~$\cdot$~--}, $\alpha=0.1$; {--~--~--}, $\alpha=0.5$; {------}, $\alpha=0.9$. }  
	\label{fig:figure4}
\end{figure}
Adjustment of $\Delta t$ is a general approach for error reduction by reducing length between time points in numerical methods. Eq.~(\ref{eqn:eq34}) makes clear that the L1-norm error of each subset is the order of $2-\alpha$ in terms of the length between time points ({\it e.g.}, $\Delta t$ in $U_0$ and $2^{l-1}\Delta t$ in $U_l$). To reduce $\Delta t$ with a given memory length $T$, $m$ is also need to be adjusted by the relation $T=m\Delta t$. In Eq.~(\ref{eqn:eq34}), the L1-norm error is not only a function of $\Delta t$ and $\alpha$ but also a function of $m$ through $A(m,\alpha)$ and $B(m,\alpha)$. Fig.~\ref{fig:figure4} shows values of $A(m,\alpha)$ and $B(m,\alpha)$ along $m$ with $\alpha=0.1,~0.5$, and $0.9$. From Fig.~\ref{fig:figure4}(a), there is an in-direct correlation between $A(m,\alpha)$ and $m$. Meanwhile, Fig.~\ref{fig:figure4}(b) shows that $B(m,\alpha)$ decreases to $-\alpha$ order in terms of $m$. This relation is also derived by a Taylor series expansion of $B(m,\alpha)$ in Appendix~\ref{app:app1}.
Thus, Eq.~(\ref{eqn:eq34}) is rewritten as follows:
\begin{equation}
\left| \frac{\mathrm{d}^{\alpha} f}{\mathrm{d}t^{\alpha}}-\left.{\frac{\mathrm{d}^{\alpha} f}{\mathrm{d}t^{\alpha}}}\right\vert_{L1}  \right| \leqslant \frac{M}{2\Gamma(3-\alpha)} \left\lbrace \Delta t^{2-\alpha}A(m,\alpha)+c(\alpha)T^{-\alpha}\sum_{l=1}^{L}(2^{l-1})^{2-\alpha}\Delta t^2  \right\rbrace,
\label{eqn:eq35}
\end{equation}
where $c(\alpha)$ is a proportional function of $B(m,\alpha)$. Now, the second term in Eq.~(\ref{eqn:eq35}) becomes $\mathcal{O}(\Delta t^{2})$. Thus, the order of accuracy changes from $\mathcal{O}(\Delta t^{2-\alpha})$ of the first term to $\mathcal{O}(\Delta t^{2})$ of the second term as time elapses. Fig.~\ref{fig:figure5}(a) shows that the L1-norm error decreases to $\mathcal{O}(\Delta t^{2-\alpha})$ at $t=2^5$, while Fig.~\ref{fig:figure5}(b) shows that the L1-norm error decreases to $\mathcal{O}(\Delta t^{2})$ at $t=2^{15}$. Fig.~\ref{fig:figure5}(c) shows the zeroth order of accuracy for the fixed memory method at $t=2^{10}$. As derived in Eq.~(\ref{eqn:eq11}), the error function of the fixed memory method is independent of $\Delta t$ when $T$ is given.
\begin{figure}
	\centering
	\includegraphics[width=0.327\linewidth]{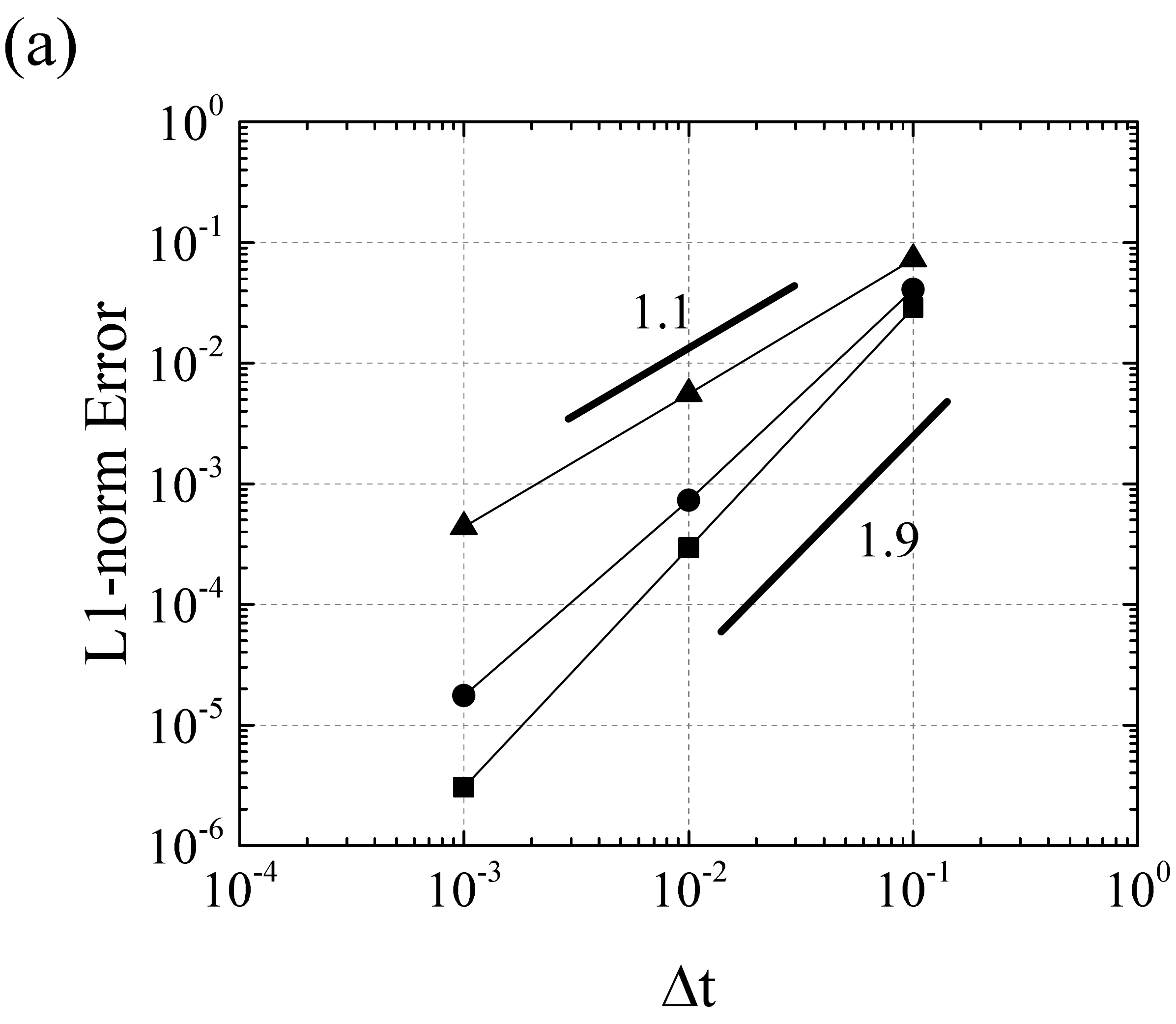}
	\includegraphics[width=0.327\linewidth]{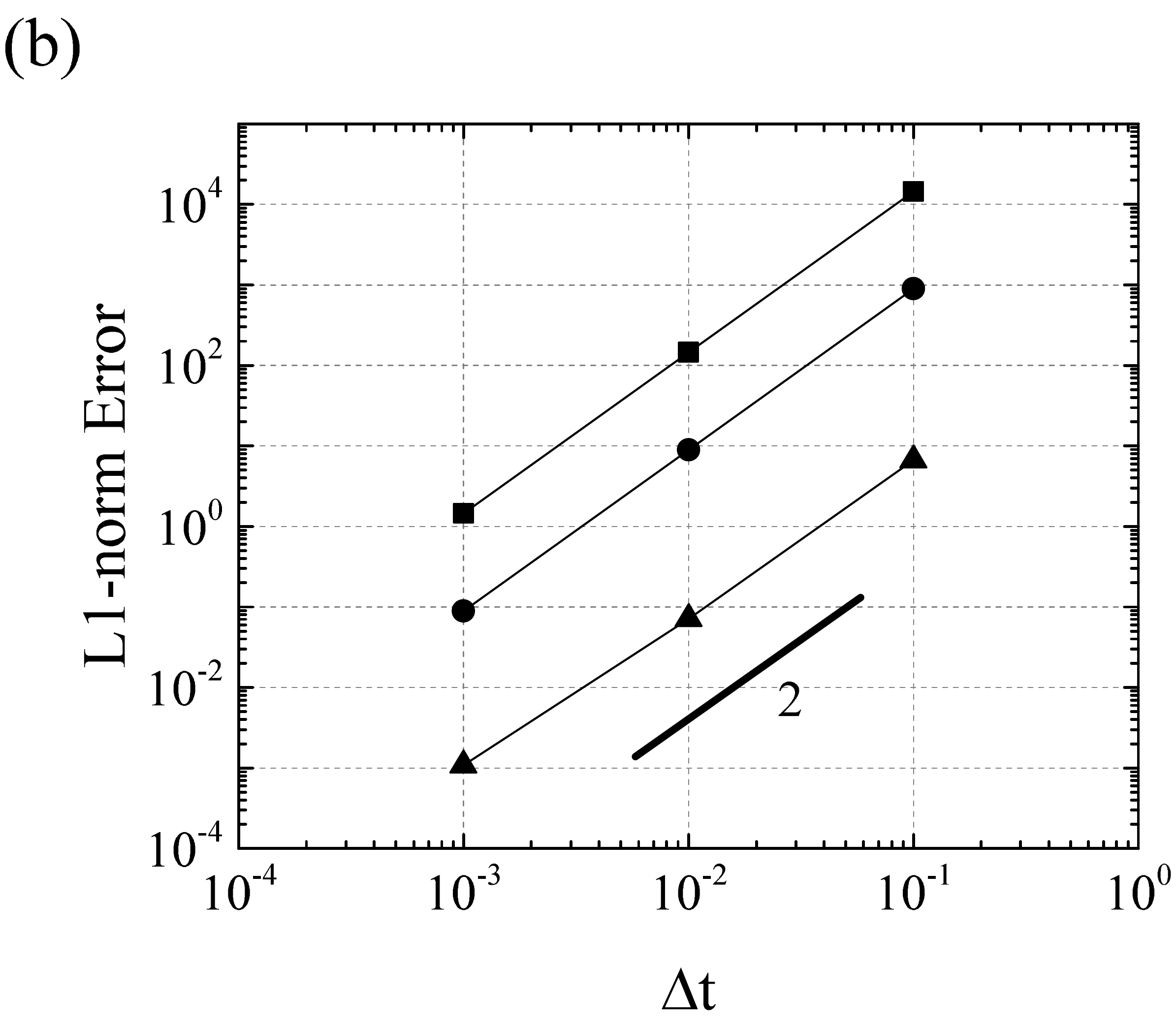}
	\includegraphics[width=0.327\linewidth]{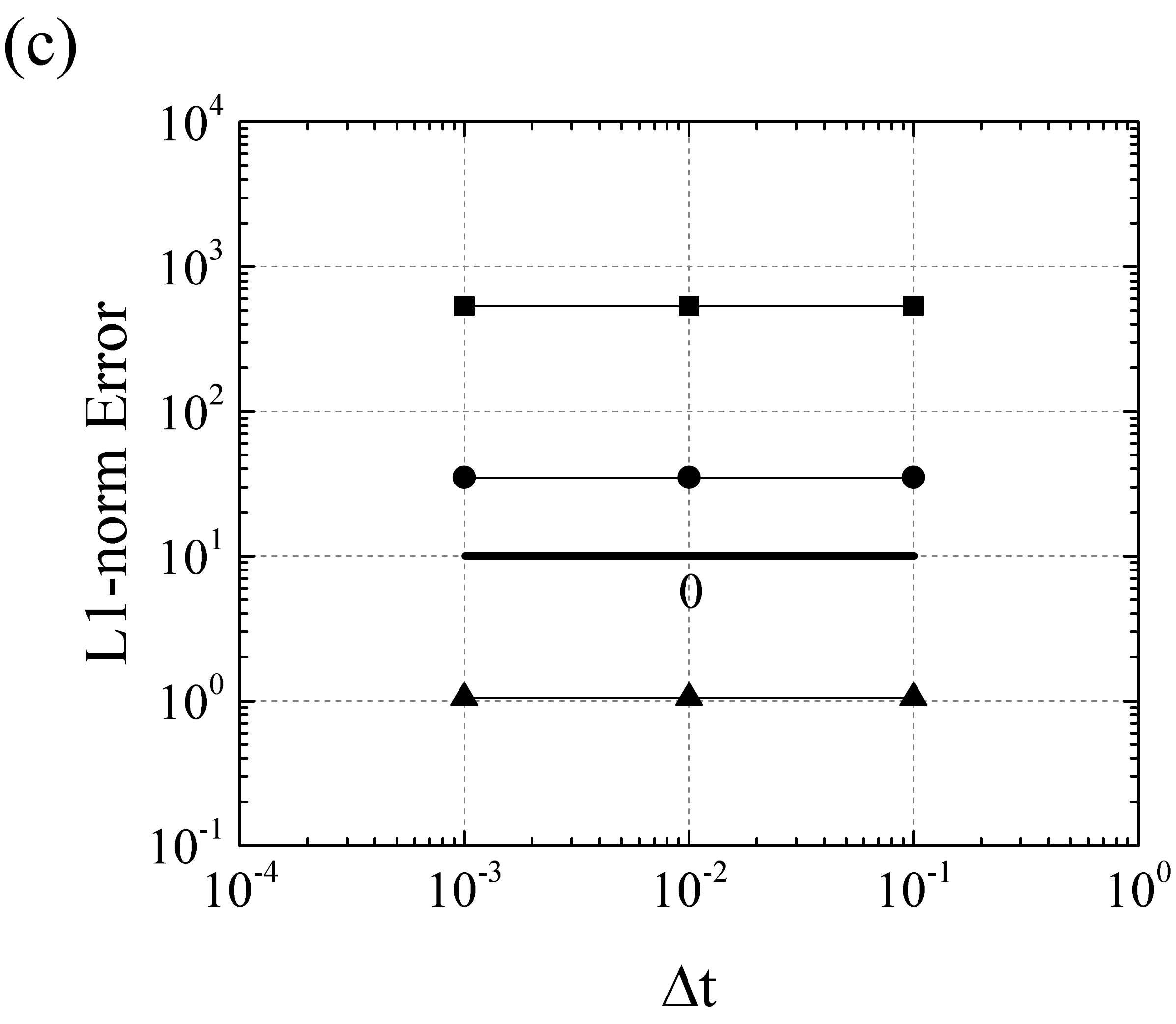}
	\caption{Plots of the order of accuracy in terms of $\Delta t$ for (a) the present adaptive memory method at $t=2^5$, (b) the present adaptive memory method at $t=2^{15}$ and (c) the fixed memory method at $t=2^{10}$. {\small{$\blacksquare$}}, $\alpha=0.1$; {\large{$\bullet$}}, $\alpha=0.5$; {\small{$\blacktriangle$}}, $\alpha=0.9$.}
	\label{fig:figure5}
\end{figure}

\section{Time-fractional diffusion equation}
\label{sec:sec5}
In this section, the accuracy and the computational cost of the present adaptive memory method are examined on a practical problem. Anomalous diffusion is a nonlinear phenomenon usually observed in diffusion process within living cells or porous media. Unlike standard diffusion, the anomalous diffusion phenomenon is represented by time-fractional diffusion with a fractional order $\alpha$. Depending on the range of a fractional order $\alpha$, it results in sub-diffusion for $0<\alpha<1$, standard diffusion for $\alpha=1$, and super-diffusion for $\alpha>1$: 
\begin{equation}
\frac{\partial^\alpha f(x,t)}{\partial t^\alpha}=\mu\frac{\partial^2 f(x,t)}{\partial x^2},
\label{eqn:eq36}
\end{equation}
where diffusion coefficient $\mu$ has a unit of $distance^2/time^{\alpha}$.

Consider, for example, a time-fractional sub-diffusion equation 
\begin{equation}
\frac{\partial^\alpha f(x,t)}{\partial t^\alpha}=\mu\frac{\partial^2 f(x,t)}{\partial x^2},~0<x<L,~ t>0,
\label{eqn:eq37}
\end{equation}
subjects to initial and Dirichlet boundary conditions,
\begin{equation}
f(x,0)=\sin(\pi x/L),~f(0,t)=f(L,t)=0.
\label{eqn:eq38}
\end{equation}
The analytic solution of Eq.~(\ref{eqn:eq37}) can be easily derived by using the methods of separation of variables and the Laplace transformation in Corollary~\ref{cor:cor1}.
Assume a solution for Eq.~(\ref{eqn:eq37}) in the form as follows:
\begin{equation}
f(x,t)=X(x)T(t).
\label{eqn:eq39}
\end{equation}
Substitution of Eq.~(\ref{eqn:eq39}) into Eq.~(\ref{eqn:eq37}) gives
\begin{equation}
XT^{(\alpha)}=\mu X^{(2)}T,
\label{eqn:eq40}
\end{equation}
where superscriptions represent differentiations with respect to variables $x$ and $t$. After being divided by $\mu X(x)T(t)$ and set equal to a constant $-\lambda^2$ for finding a nontrivial solution, it becomes
\begin{equation}
\frac{T^{(\alpha)}}{\mu T}=\frac{X^{(2)}}{X}=-\lambda^2.
\label{eqn:eq41}
\end{equation}
By separating each variable,
\begin{equation}
X^{(2)}+\lambda^2 X=0,
\label{eqn:eq42}
\end{equation}
and
\begin{equation}
T^{(\alpha)}+\mu\lambda^2 T=0.
\label{eqn:eq43}
\end{equation}
The general solution of Eq.~(\ref{eqn:eq42}) is given as follows:
\begin{equation}
X(x)=a\sin{\lambda x}+b\cos{\lambda x}.
\label{eqn:eq44}
\end{equation}
To satisfy the boundary conditions, $b=0$ where $x=0$ and $\lambda=n\pi/L,~n\geq1$ where $x=L$. Consequently, the corresponding solutions become
\begin{equation}
X_n(x)=a_n\sin{\frac{n\pi x}{L} },~n\geq 1.
\label{eqn:eq45}
\end{equation}
Also, the general solution of Eq.~(\ref{eqn:eq43}) is obtained by using the Laplace transformation. From Corollary~\ref{cor:cor1}, Eq.~(\ref{eqn:eq43}) is transformed as follows:
\begin{equation}
s^\alpha \mathcal{L}\{T\}(s) -s^{\alpha-1}T(0)+\mu\lambda^2\mathcal{L}\{T\}(s)=0.
\label{eqn:eq46}
\end{equation}
Then, it becomes
\begin{equation}
\mathcal{L}\{T\}(s) = T(0)\frac{s^{\alpha-1}}{s^\alpha+\mu\lambda^2}.
\label{eqn:eq47}
\end{equation}
By applying inverse Laplace transformation \cite{Mathai2008},
\begin{equation}
\mathcal{L}^{-1}\left(\frac{s^{\alpha-1}}{s^\alpha+\gamma}\right)=E_{\alpha,1}(-\gamma t^\alpha),
\label{eqn:eq48}
\end{equation}
a general solution for $T(t)$ is determined as follows:
\begin{equation}
T(t)=T(0)E_{\alpha,1}(-\mu \lambda^2 t^\alpha),
\label{eqn:eq49}
\end{equation}
where Mittag-Leffler function $E_{\alpha,1}$ is defined by an infinite sum of
\begin{equation}
E_{\alpha,1}(z)=\sum_{k=0}^{\infty}\frac{z^k}{\Gamma(\alpha k +1)}.
\label{eqn:eq50}
\end{equation}
Substituting $\lambda=n\pi/L$ into Eq.~(\ref{eqn:eq49}) leads to
\begin{equation}
T_n(t)=T(0)E_{\alpha,1}(-\mu (n\pi/L)^2 t^\alpha),~n\geq 1.
\label{eqn:eq51}
\end{equation}
Thus, the following sequence of solutions are obtained,
\begin{equation}
f_n(x,t)=X_n(x)T_n(t)=c_n\sin\frac{n\pi x}{L}E_{\alpha,1}(-\mu (n\pi/L)^2 t^\alpha),~n\geq 1.
\label{eqn:eq52}
\end{equation}
By taking linear combinations of $f_n(x,t)$, $f(x,t)$ is represented as follows:
\begin{equation}
f(x,t)=\sum_{n=1}^{\infty}f_n(x,t)=\sum_{n=1}^{\infty}c_n\sin\frac{n\pi x}{L}E_{\alpha,1}(-\mu (n\pi/L)^2 t^\alpha).
\label{eqn:eq53}
\end{equation}
By the initial condition of $f(x,0)=\sin(\pi x/L)$, $c_1=1$ and $c_n=0$, $n>1$. Finally, the analytic solution of Eq.~(\ref{eqn:eq37}) is determined as follows:
\begin{equation}
f(x,t)=\sin\frac{\pi x}{L}E_{\alpha,1}(-\mu (\pi/L)^2 t^\alpha).
\label{eqn:eq54}
\end{equation}
Note that Eq.~(\ref{eqn:eq54}) recovers the analytic solution of a standard diffusion equation when $\alpha=1$.

The time-fractional diffusion equation is numerically solved to assess the efficacy of the fixed memory method, the adaptive memory method of MacDonald~\textit{et al.}, and the present adaptive memory method. Eq.~(\ref{eqn:eq37}) is discretized by the implicit L1 scheme in time and the central difference scheme in space:
\begin{equation}
\frac{1}{\Gamma(1-\alpha)}\sum_{k=0}^{n}\omega_t^k \frac{f_i^{k+1}-f_i^{k}}{\Delta t^k}=\mu\frac{f_{i+1}^{n+1}-2f_{i}^{n+1}+f_{i-1}^{n+1}}{\Delta x^2}.
\label{eqn:eq55}
\end{equation}
To build a tridiagonal matrix, Eq.~(\ref{eqn:eq55}) can be rearranged as follows:
\begin{equation}
\begin{split}
&-\mu\frac{\Delta t^n}{\Delta x^2}f_{i+1}^{n+1}+\left(\frac{1}{\Gamma(1-\alpha)}\omega_t^n+2\mu\frac{\Delta t^n}{\Delta x^2}\right)f_{i}^{n+1}-\mu\frac{\Delta t^n}{\Delta x^2}f_{i-1}^{n+1}=\\
&\frac{1}{\Gamma(1-\alpha)}\omega_t^n f_{i}^{n}-\frac{\Delta t^n}{\Gamma(1-\alpha)}\sum_{k=0}^{n-1}\omega_t^k \frac{f_i^{k+1}-f_i^{k}}{\Delta t^k}.\\
\end{split}
\label{eqn:eq56}
\end{equation}
This tridiagonal matrix is solved by the Thomas algorithm. Note that $\Delta t^k$ between $t^{k+1}$ and $t^k$ is uniform for the full and fixed memory methods but it can be non-uniform for the present adaptive memory method. 

In case of the adaptive memory method of MacDonald~\textit{et al.} based on the Gr{\"u}wald-Lenikov derivative, Eq.~(\ref{eqn:eq37}) is discretized by the implicit Gr{\"u}wald-Lenikov formula in time and the central difference scheme in space as follows:
\begin{equation}
\frac{1}{\Delta t^{\alpha}}\sum_{k=0}^{n}\omega_n^k(f_i^{k+1}-f_i^0)=\mu\frac{f_{i+1}^{n+1}-2f_{i}^{n+1}+f_{i-1}^{n+1}}{\Delta x^2},
\label{eqn:eq57}
\end{equation}
where
\begin{equation}
w_n^k=(-1)^{n-k}\left(\begin{array}{c}\alpha\\ n-k\end{array}\right)=(-1)^{n-k}\frac{\Gamma(\alpha+1)}{(n-k)!\Gamma(\alpha-n+k+1)}.
\label{eqn:eq58}
\end{equation}
Note that $\omega_n^k$ is scaled by the adaptive power-law algorithm of MacDonald~\textit{et al.} to compensate skipped time points as follows:
\begin{equation}
W_n^k=w_n^k \left( \frac{t^{k+1}-t^{k}}{\Delta t} \right).
\label{eqn:eq59}
\end{equation}
To build a tridiagonal matrix, Eq.~(\ref{eqn:eq57}) can be rearranged as follows:
\begin{equation}
-\mu\frac{\Delta t^{\alpha}}{\Delta x^2}f_{i+1}^{n+1}+\left(W_n^n+2\mu\frac{\Delta t^{\alpha}}{\Delta x^2} \right)f_{i}^{n+1}-\mu\frac{\Delta t^{\alpha}}{\Delta x^2}f_{i-1}^{n+1}=W_n^{n}f_i^0-\sum_{k=0}^{n-1}W_n^k(f_i^{k+1}-f_i^0).
\label{eqn:eq60}
\end{equation}
This tridiagonal matrix is also solved by the Thomas algorithm.

\begin{figure}
	\centering
	\includegraphics[width=0.48\linewidth]{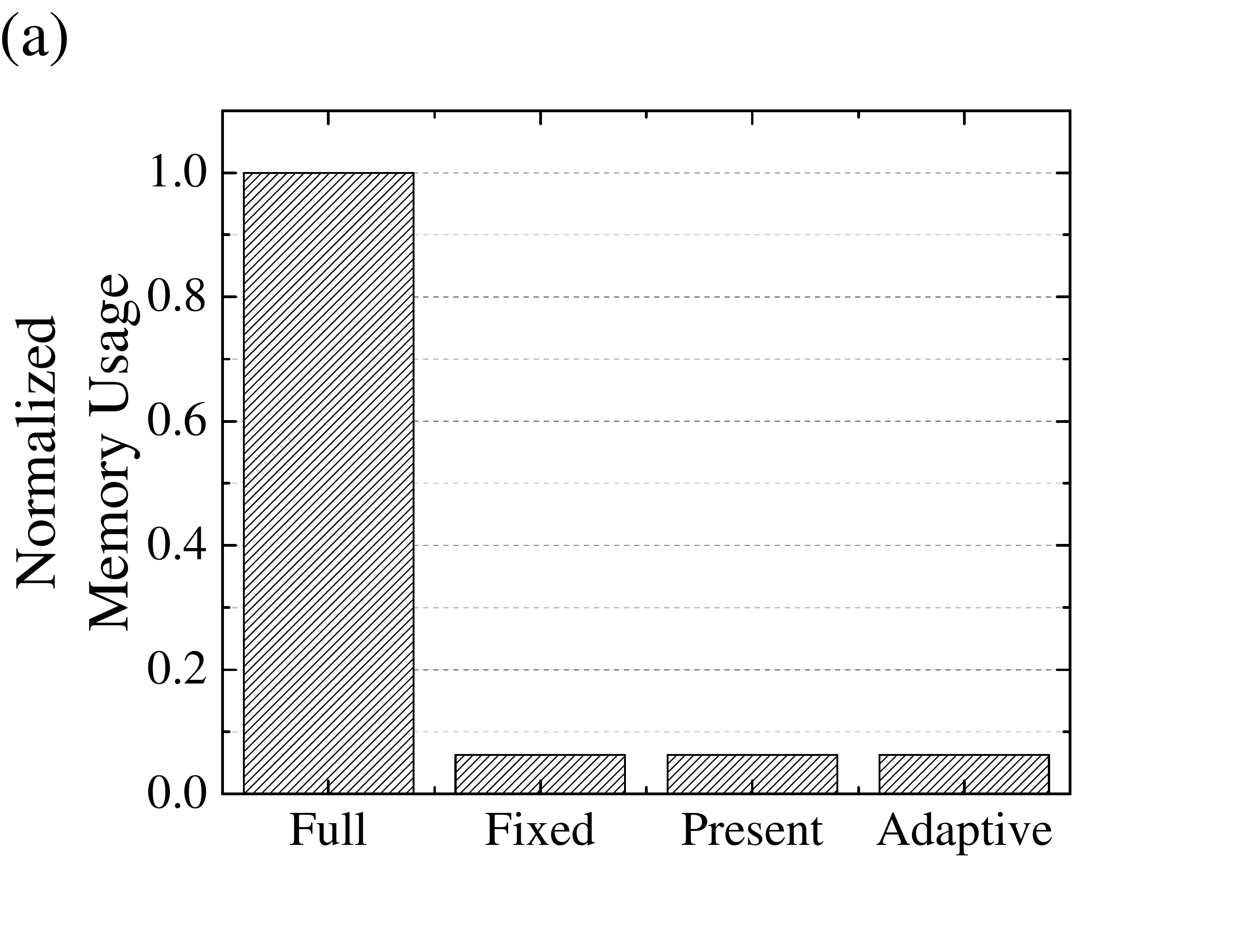}
	\includegraphics[width=0.48\linewidth]{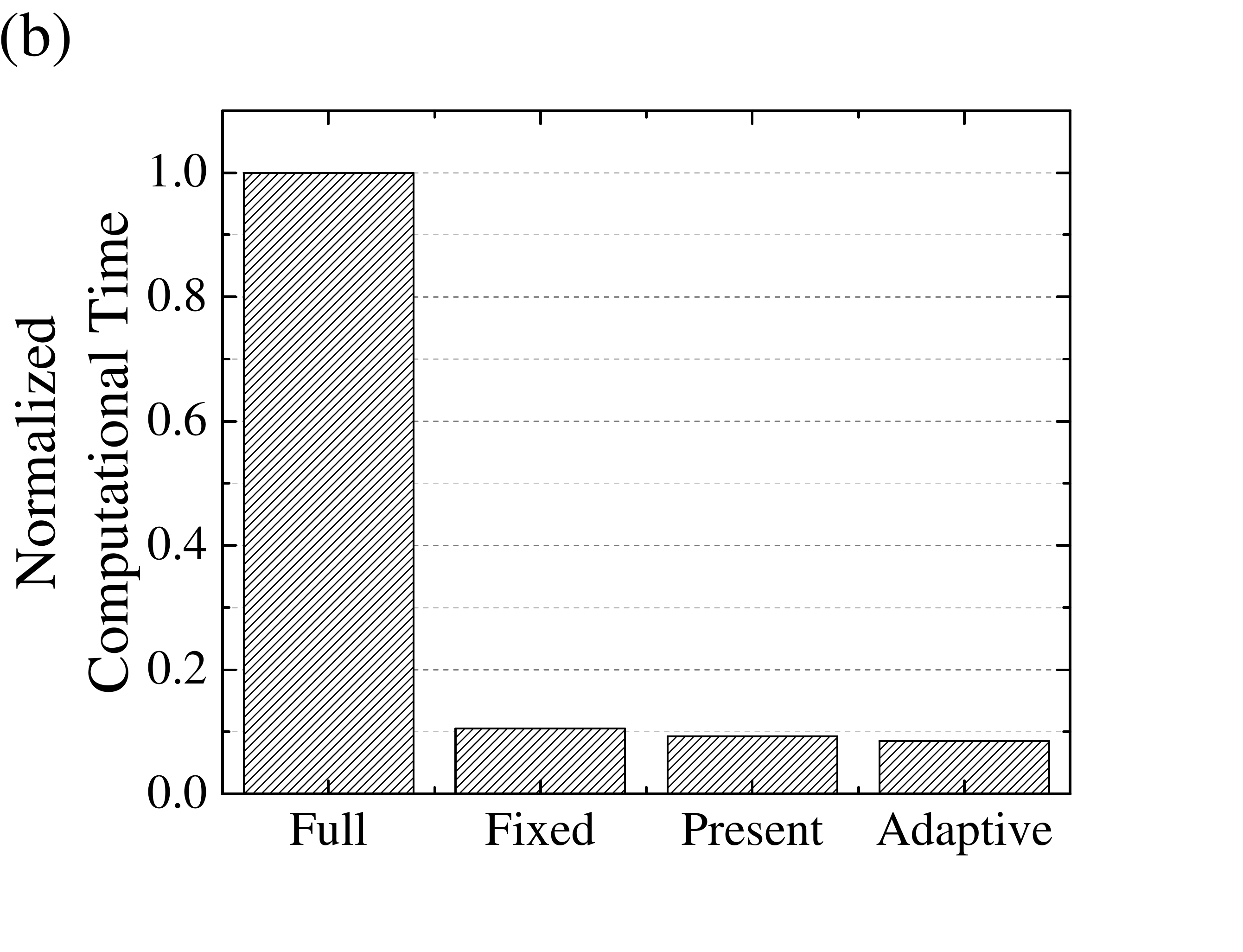}
	\caption{Computational costs of each memory method for the time-fractional sub-diffusion equation at $t=12.8$. (a) Normalized computational memory required for each method, and (b) normalized computational time required for each method.}
	\label{fig:figure6}
\end{figure}
\begin{figure}
	\centering
	\includegraphics[width=0.45\linewidth]{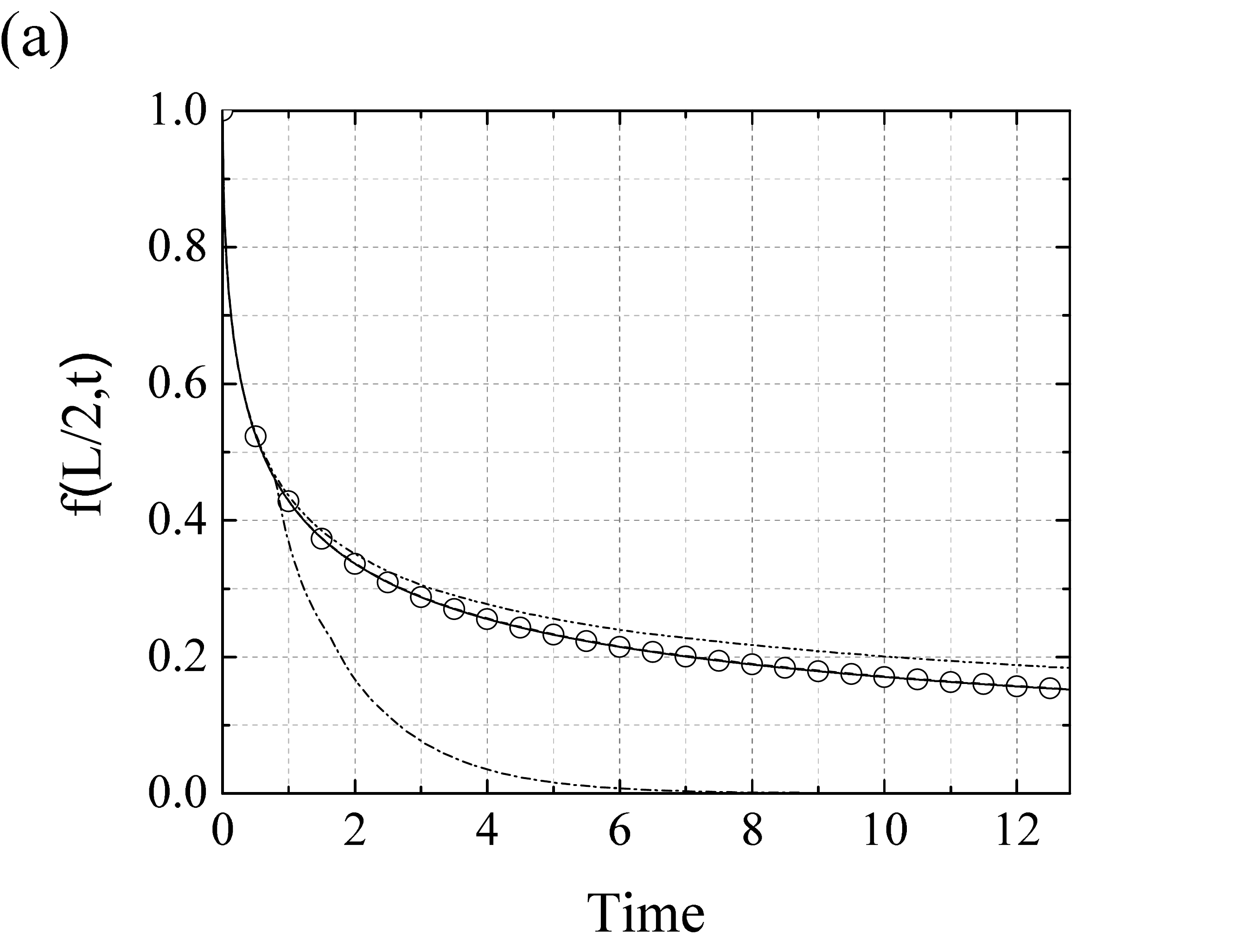}
	\includegraphics[width=0.45\linewidth]{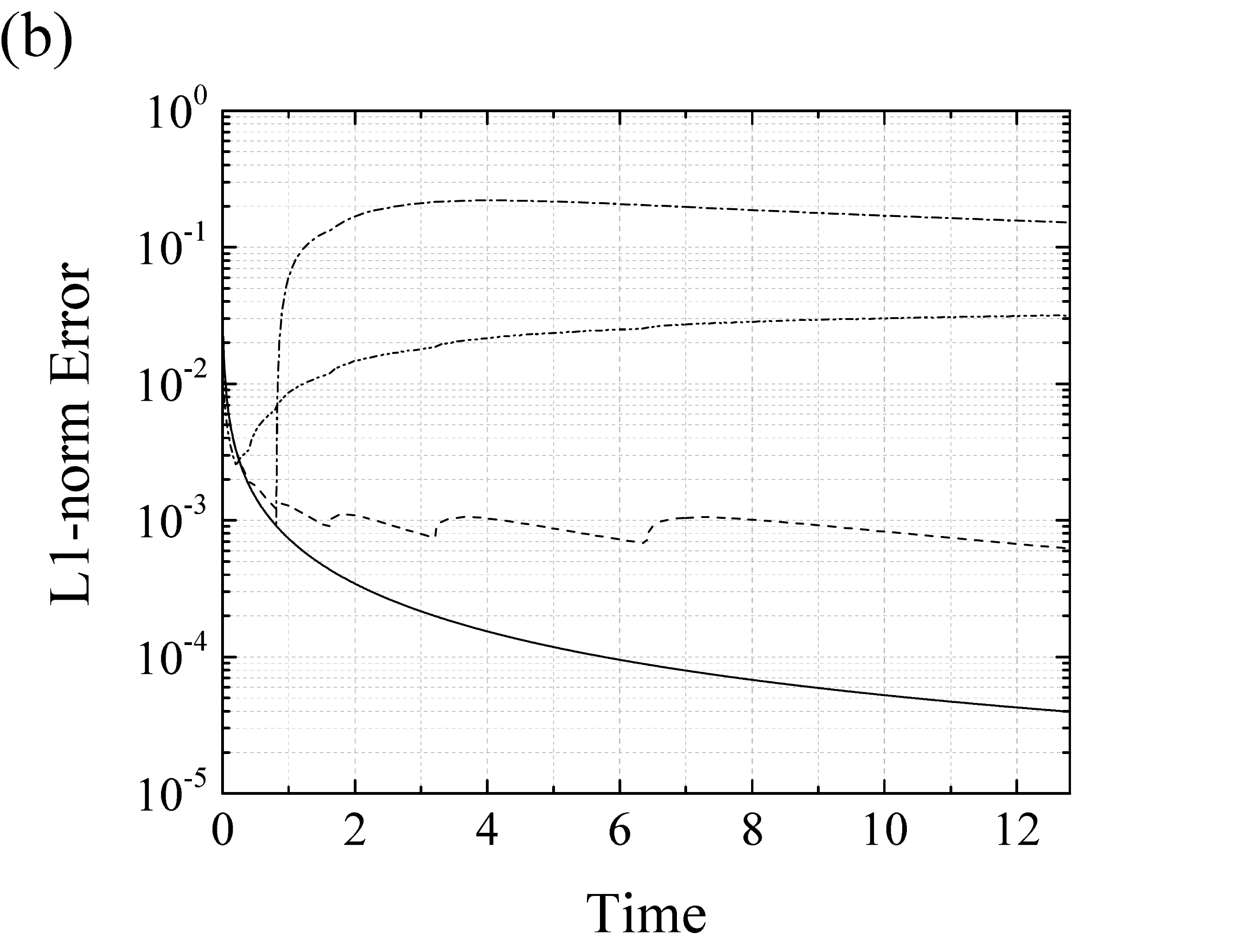}
	\caption{Numerical results of each memory method for the time-fractional sub-diffusion equation with $\alpha=0.5$. (a) Exact and numerical solutions at $x=L/2$ as a function of time and (b) L1-norm error as a function of time. {\large{$\circ$}}, the analytic solution; {------}, the full memory method; {--~--~--}, the present adaptive memory method; {--~$\cdot$~--}, the fixed memory method; {--~$\cdot$~$\cdot$~--}, the adaptive memory method of MacDonald \textit{et al.}.}
	\label{fig:figure7}
\end{figure}
Using the analytic solution of a time-fractional diffusion equation, the accuracy and computational benefits of each method are compared quantitatively. The full memory method, the fixed memory method, and the present adaptive memory method are applied to Eq.~(\ref{eqn:eq55}), and the adaptive memory method of MacDonald~\textit{et al.} is applied to Eq.~(\ref{eqn:eq57}) with $\alpha=0.5$, $L=10$, $\Delta x=0.1$, $\Delta t=0.01$, and $\mu=(L/\pi)^2$. The two adaptive memory methods have their own memory length $T=0.1$, then corresponding number of time points are allowed to be stored in the fixed memory method for fair comparison. Fig.~\ref{fig:figure6}(a) shows the maximum memory usage of each method until $t=12.8$. The fixed and adaptive memory methods are found to reduce the memory usage to less than 10\% of the memory usage for the full memory method. Consequently, the computational time is also reduced in a similar order as shown in Fig.~\ref{fig:figure6}(b). However, the calculated function values show noticeable difference. The fixed memory method starts to severely underestimate the function values after its memory length. The adaptive memory method of MacDonald~\textit{et al.} also overestimates the values after $t=2T$. The full memory method and the present adaptive memory method show much better estimation compared to the exact solution as shown in Fig.~\ref{fig:figure7}(a). Again, Fig.~\ref{fig:figure7}(b) shows L1-norm errors for each method as a function of time in detail. The superior accuracy of the present adaptive memory method is revealed compared to other methods clearly except for the full memory method which is impractically expensive. Interestingly, errors of the present memory method and the adaptive memory method of MacDonald~\textit{et al.} show little bumps because the maximum length of neighbor time points is extended by the power-law algorithm at $t=2^lT,~l\in\mathbb{N}$.

\begin{figure}
	\centering
	\includegraphics[width=0.45\linewidth]{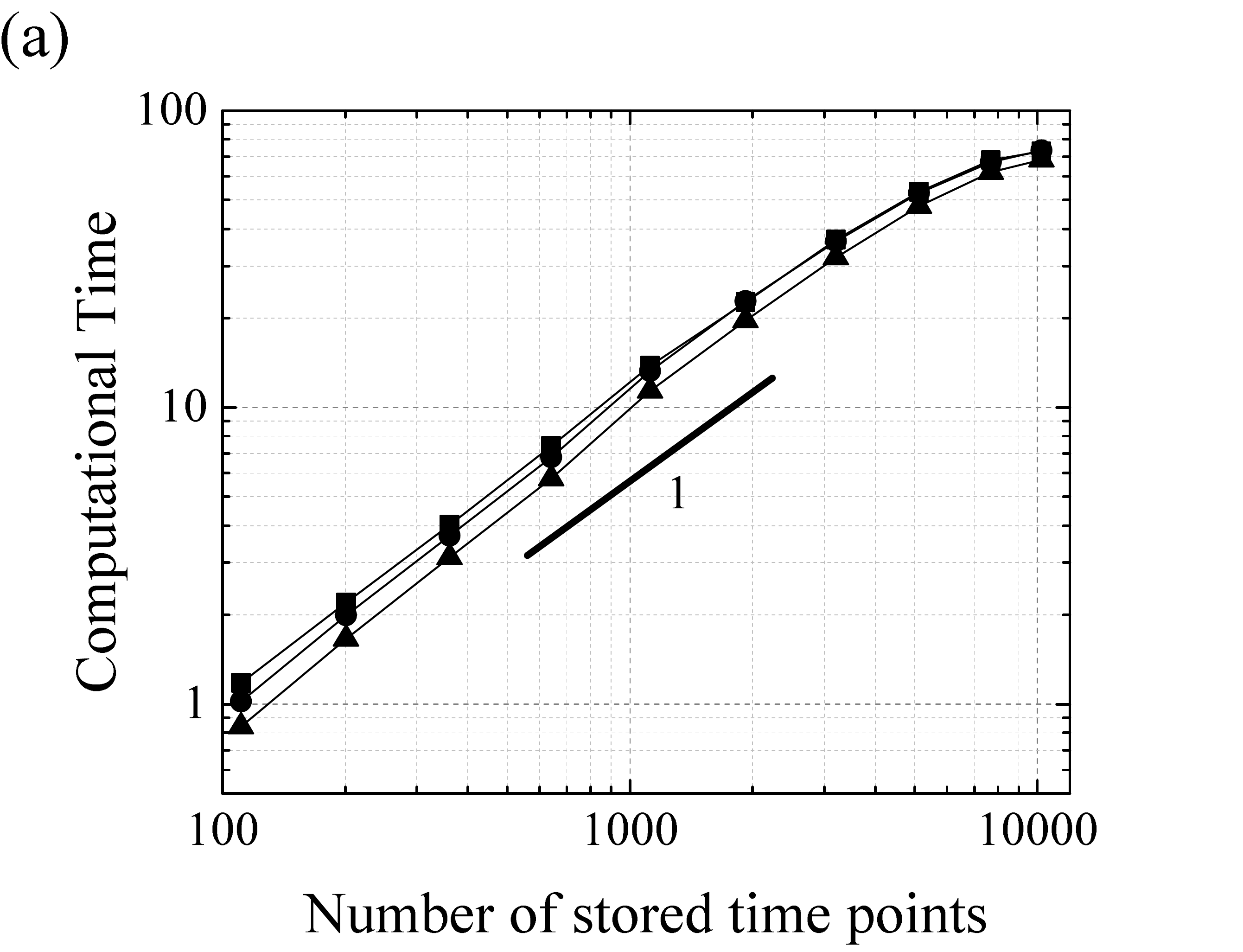}
	\includegraphics[width=0.45\linewidth]{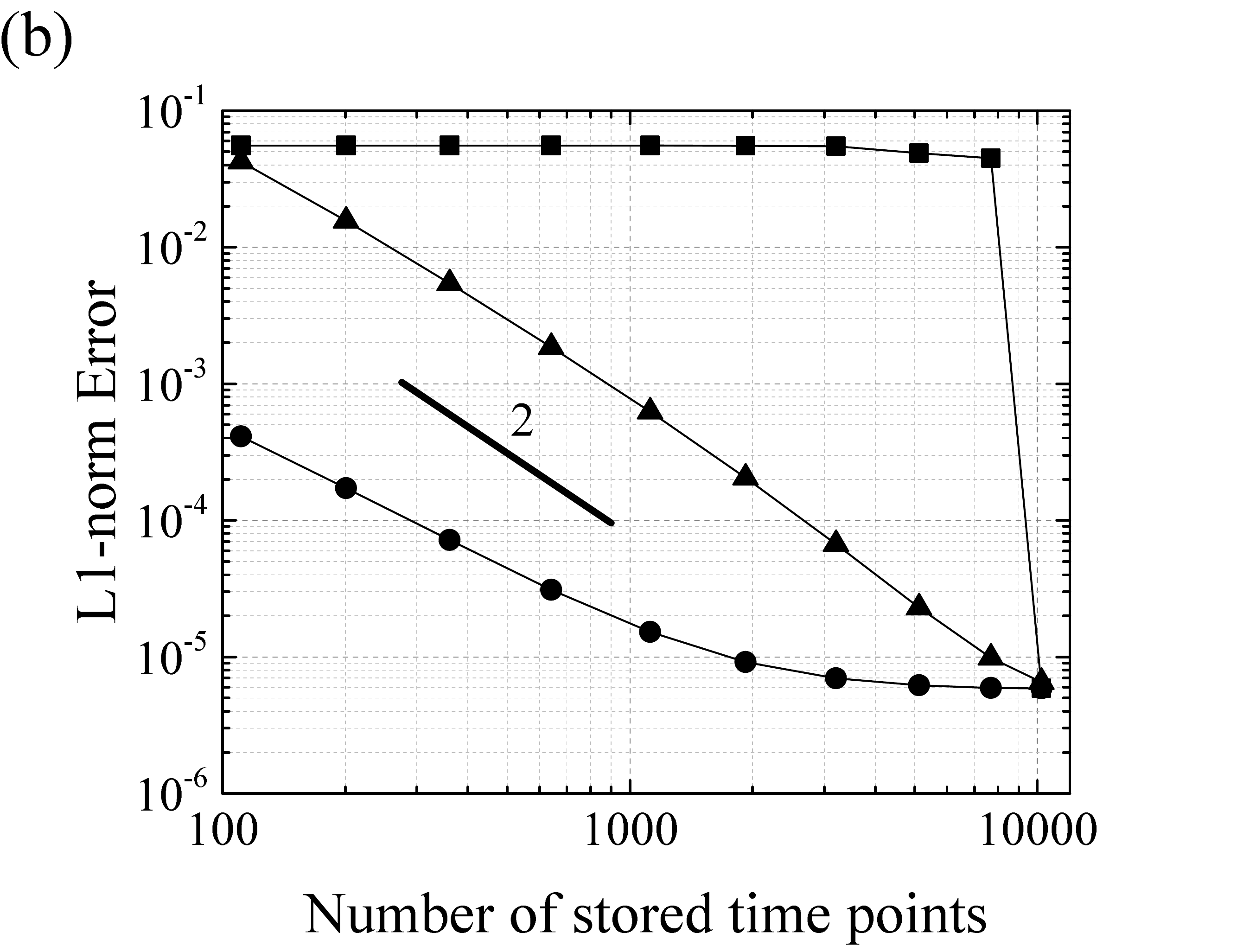}
	\caption{Numerical efficacy assessed by changing the number of stored time points with $\alpha=0.5$ at $t=102.4$. (a) Computational time for each method and (b) L1-norm error for each method. {\large{$\bullet$}}, the present adaptive memory method; {\small{$\blacksquare$}}, the fixed memory method; {\small{$\blacktriangle$}}, the adaptive memory method of MacDonald \textit{et al.}.}
	\label{fig:figure8}
\end{figure}
For further investigation of efficacy of the present adaptive memory method, the computational time and the L1-norm error for each method are measured according to the number of stored time points. The same time-fractional sub-diffusion equation is solved until $t=102.4$. Then, the estimated function value and the computational time are compared for each method with various memory lengths for changing memory requirement of stored time points. Fig.~\ref{fig:figure8}(a) represents the relation between the computational time and the number of stored time points. All three methods require similar computational time, which is linear to the number of stored time points. However, in terms of accuracy, the present memory method preserve the accuracy much better than other methods when the number of stored time points is decreased as represented in Fig.~\ref{fig:figure8}(b).

\section{A fractional Kelvin-Voigt model}
\label{sec:sec6}
\begin{figure}
	\centering
	\includegraphics[width=0.495\linewidth]{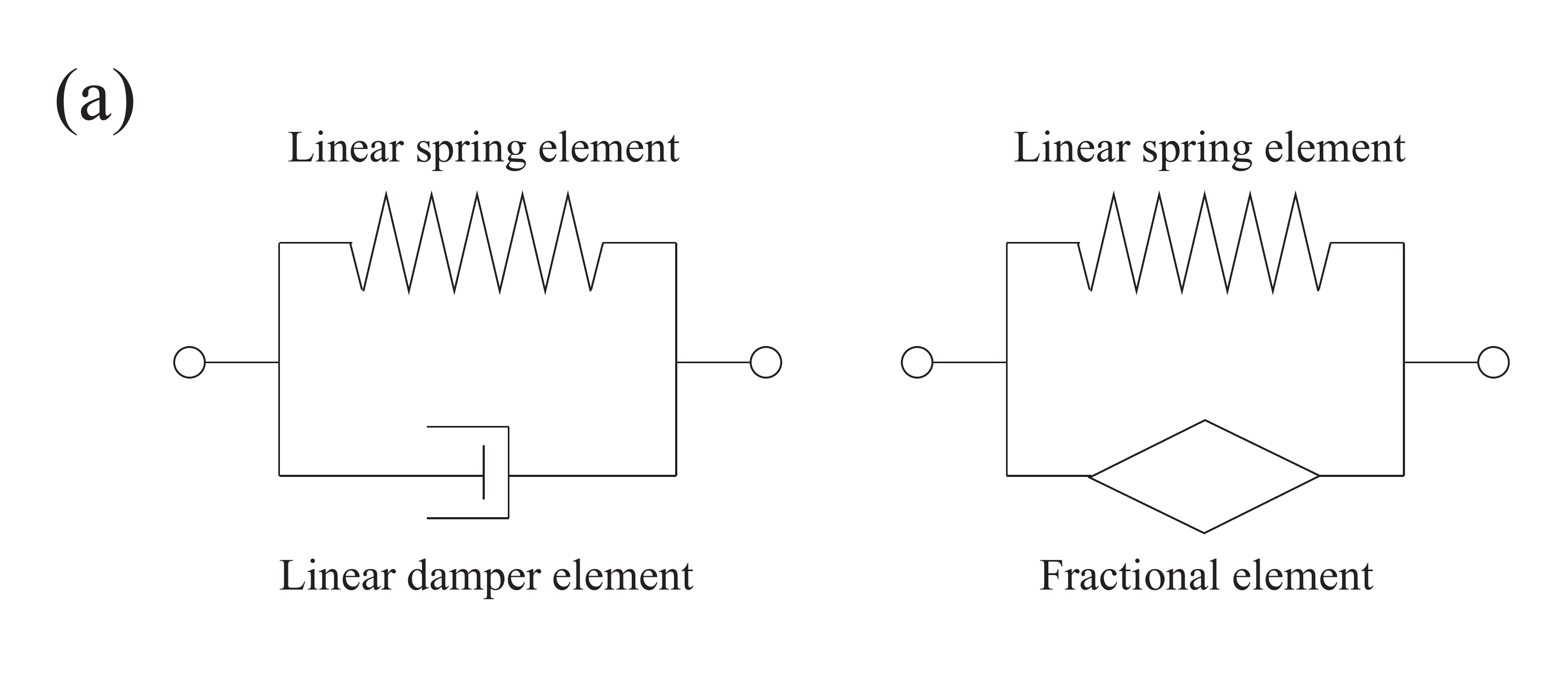}
	\includegraphics[width=0.495\linewidth]{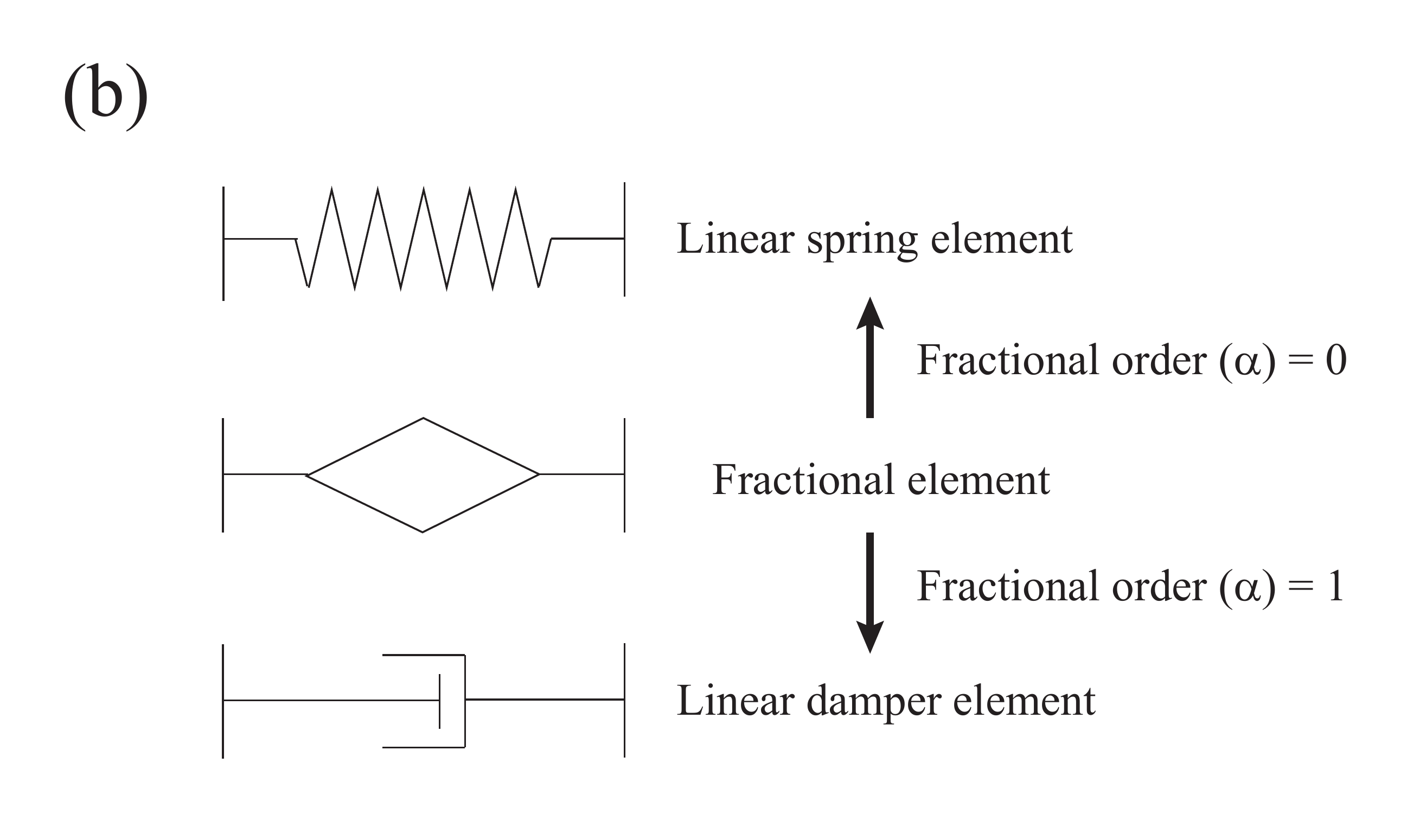}
	\caption{(a) Schematic illustration of the classical Kelvin-Voigt model (left) and the fractional Kelvin-Voigt model (right), and (b) the limiting property of the fractional element, which is reduced to a linear spring when $\alpha=0$ and a linear damper when $\alpha=1$.}
	\label{fig:figure9}
\end{figure}
Fractional viscoelasticity is one of the major applications of the fractional calculus. In the classical viscoelasticity theory, a Kelvin-Voigt model is considered to be one of the simplest models describing linear viscoelastic behaviors. The model is composed of a linear spring and a viscous damper in a parallel connection as shown in the left side of Fig.~\ref{fig:figure9}(a). When a constant loading is applied to the Kelvin-Voigt model, it shows time delayed elongation as an exponential function of time. 

To introduce the fractional calculus, the first derivative of the damper in the Kelvin-Voigt model is replaced with a fractional derivative of the fractional element for $0<\alpha<1$ as shown in the right side of Fig.~\ref{fig:figure9}(a). This fractional element is reduced to a linear spring element when $\alpha=0$ and a linear damper element when $\alpha=1$ as shown in Fig.~\ref{fig:figure9}(b). Then, the constitutive equation for a fractional Kelvin-Voigt model with a constant loading $f$ is defined as follows:
\begin{equation}
\eta\frac{\mathrm{d}^\alpha x(t)}{\mathrm{d}t^\alpha}+kx(t)=f,
\label{eqn:eq61} 
\end{equation}
where $\eta$ and $k$ is a damping constant and a spring constant, respectively and $x(t)$ is the elongation length as a function of time.
Using the Laplace transformation of the Caputo fractional derivative (see Corollary~\ref{cor:cor1}) with an initial condition $x(0)=0$, Eq.~(\ref{eqn:eq61}) is transformed as follows:
\begin{equation}
\eta s^\alpha X(s)+k X(s)=f/s.
\label{eqn:eq62}
\end{equation}
By letting $\tau^\alpha=\eta/k$, the above equation is recast as follows:
\begin{equation}
X(s)=\frac{f}{k}\left(\frac{1}{s}-\frac{s^{\alpha-1}}{s^\alpha+\tau^{-\alpha}}\right).
\label{eqn:eq63}
\end{equation}
Applying inverse Laplace transformation~\cite{Mathai2008}
\begin{equation}
\mathcal{L}^{-1}\left(\frac{s^{\alpha-1}}{s^\alpha+\lambda}\right)=E_\alpha(-\lambda t^\alpha),
\label{eqn:eq64}
\end{equation}
an analytic solution for the creep response of a fractional Kelvin-Voigt model is determined as follows:
\begin{equation}
x(t)=\frac{f}{k}\left[1-E_\alpha(-(t/\tau)^\alpha)\right],
\label{eqn:eq65}
\end{equation}
where Mittag-Leffler function $E_\alpha$ is defined as an infinite sum:
\begin{equation}
E_\alpha(z)=\sum_{k=0}^{\infty}\frac{z^k}{\Gamma(\alpha k +1)}.
\label{eqn:eq66}
\end{equation}
Note that Eq.~(\ref{eqn:eq65}) recovers the creep response of a classical Kelvin-Voigt model for $\alpha=1$.

\begin{figure}
	\centering
	\includegraphics[width=0.48\linewidth]{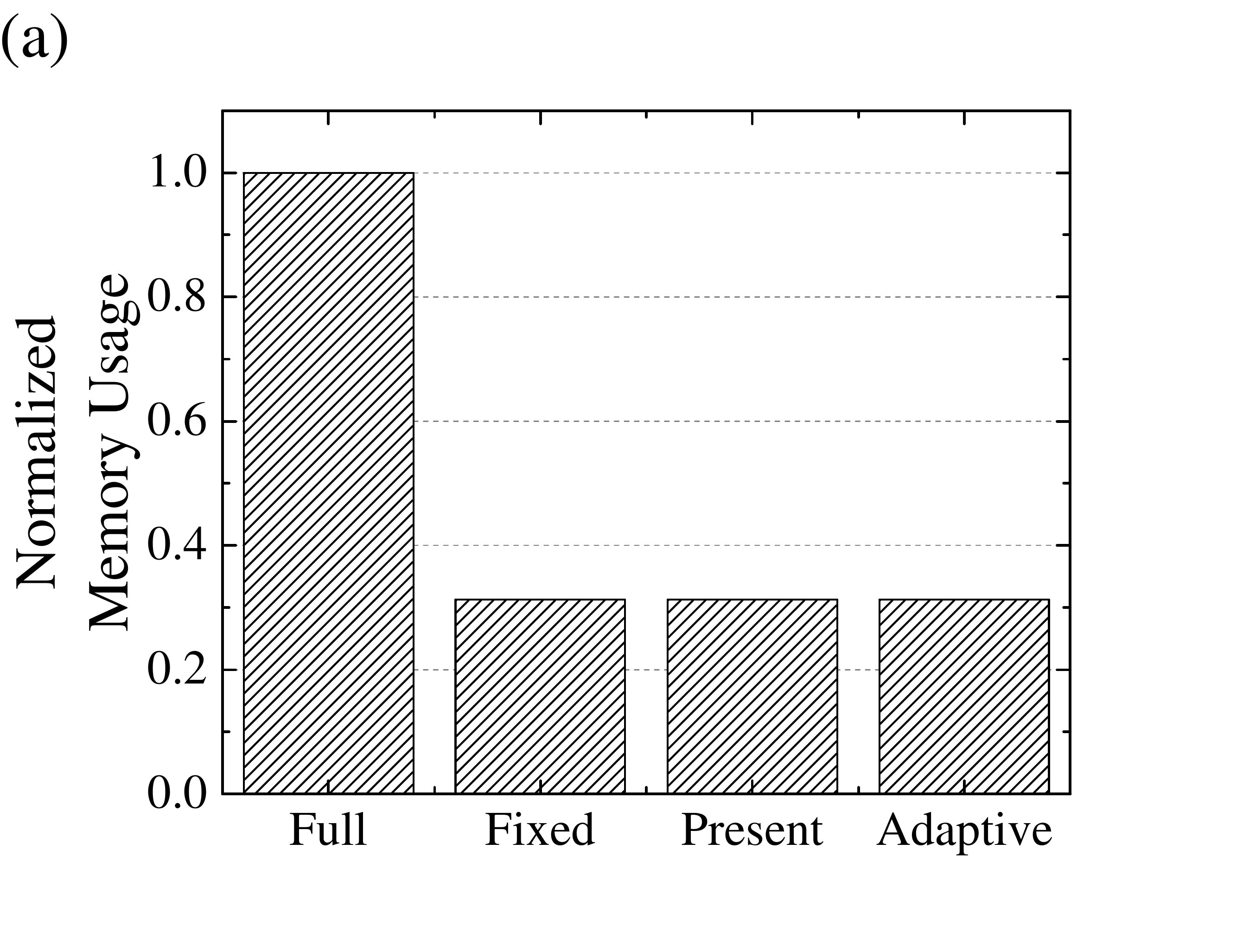}
	\includegraphics[width=0.48\linewidth]{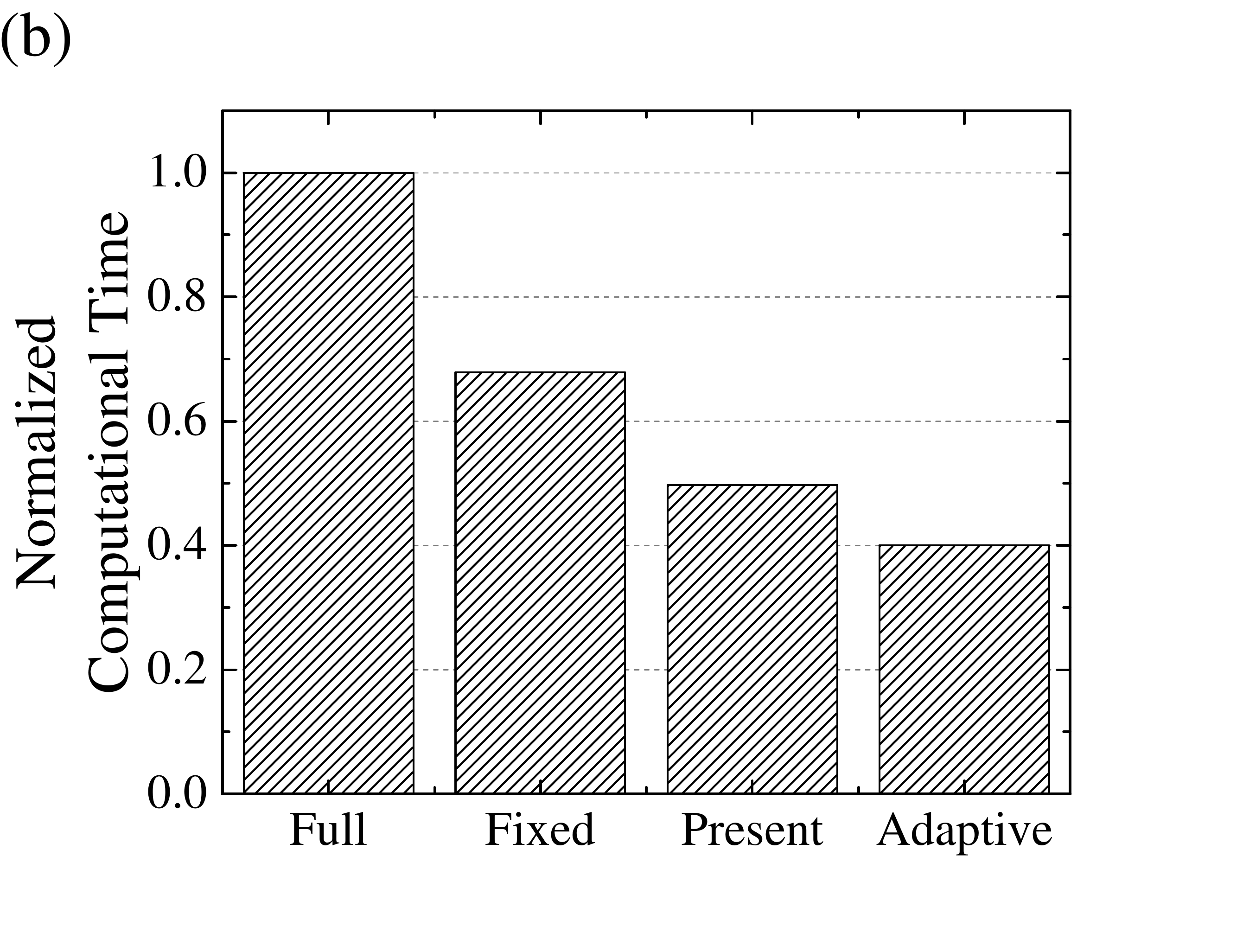}
	\caption{Computational costs of each memory method for the simulation of creep response predicted by the fractional Kelvin-Voigt model at $t=16$. (a) Normalized computational memory required for each method and (b) normalized computational time required for each method.}
	\label{fig:figure10}
\end{figure}
$\Delta t$, $T$, $f$, $\eta$, $k$, and $\alpha$ are set to $0.01\mathrm{s}$, $1\mathrm{s}$, $1\mathrm{N}$, $1\mathrm{N\cdot s^{\alpha}/m}$, $1\mathrm{N/m}$, and $0.5$, respectively. 
The corresponding number of time points are allowed to be stored in the fixed memory method for fair comparison with the results of two adaptive memory methods. Fig.~\ref{fig:figure10}(a) shows the maximum memory usage of each method until $t=16$. The fixed memory and two adaptive memory methods require the same amount of memory, which is a third of requirement for the full memory method. Consequently, the computational time is also reduced similarly as shown in Fig.~\ref{fig:figure10}(b).

\begin{figure}
	\centering
	\includegraphics[width=0.45\linewidth]{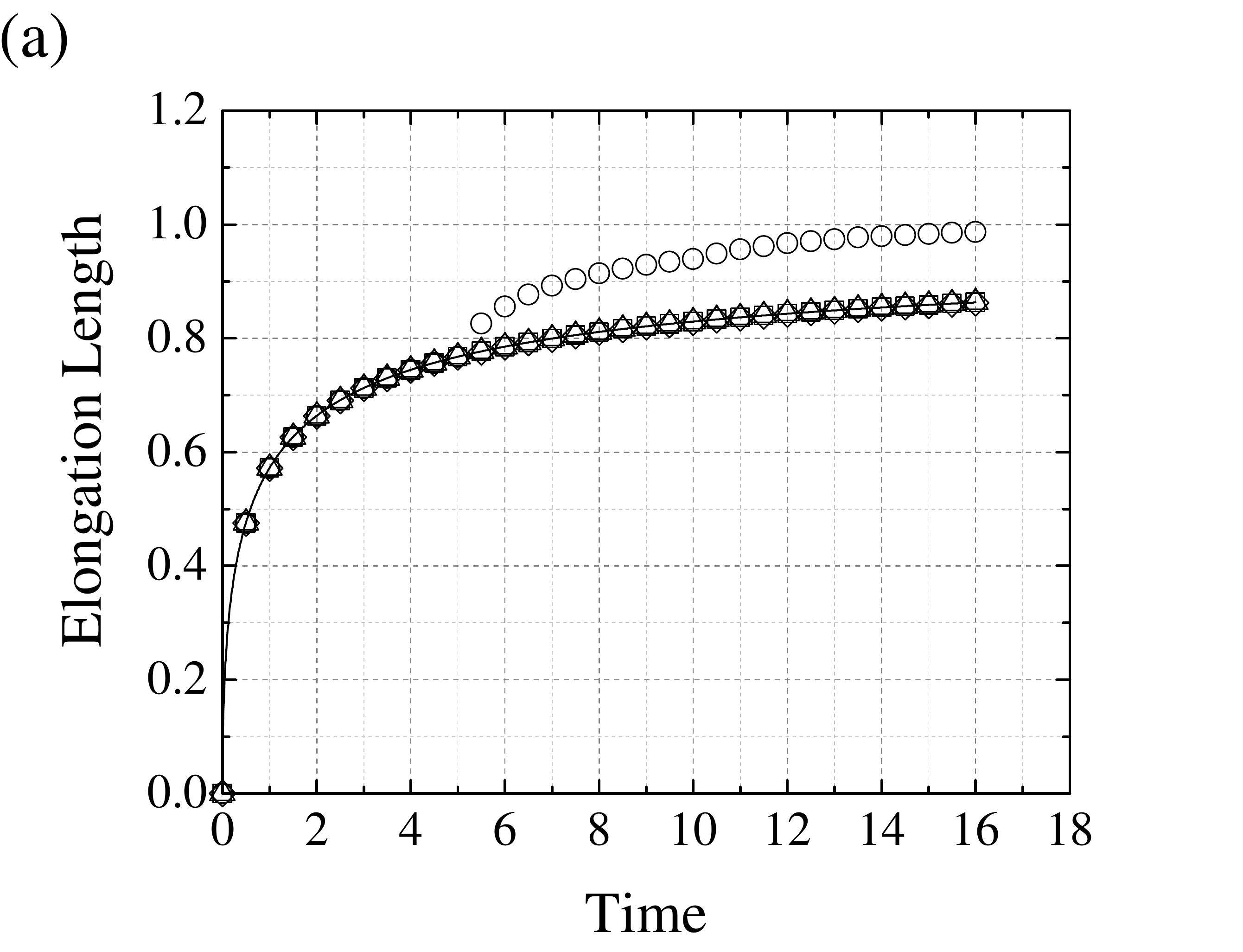}
	\includegraphics[width=0.45\linewidth]{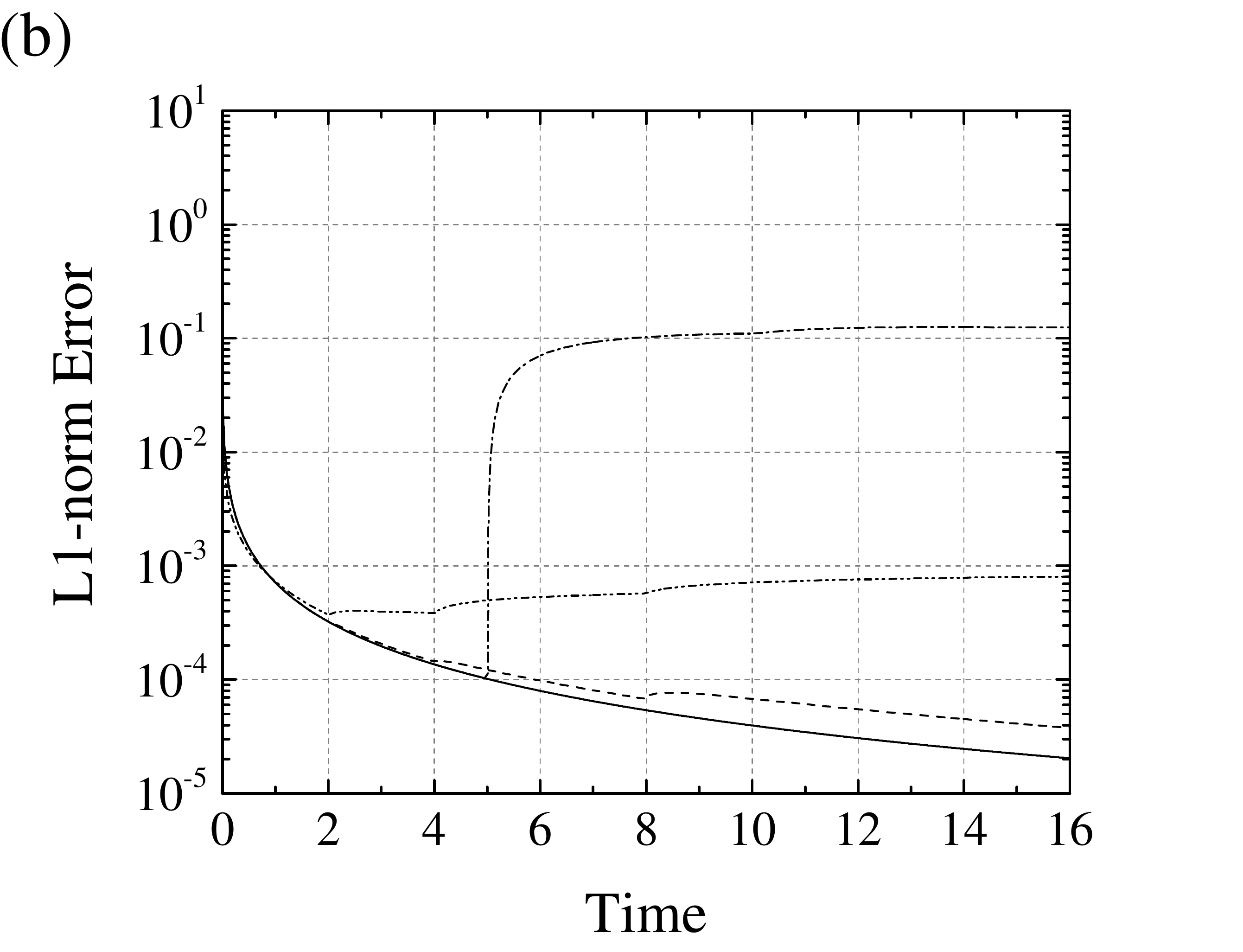}
	\caption{Simulation results for the creep response predicted by the fractional Kelvin-Voigt model. (a) Elongation length with $\alpha=0.5$ as a function of time. {------}, analytic solution; {$\triangle$}, solution by the present adaptive memory method; {\large{$\circ$}}, solution by the fixed memory method; {$\Diamond$}, solution by the adaptive memory method of MacDonald \textit{et al.} \cite{MacDonald2015}; and {$\square$}, solution by the full memory method, and (b) L1-norm errors as a function of time. {------}, the full memory method; {--~--~--}, the present adaptive memory method; {--~$\cdot$~--}, the fixed memory method; and {--~$\cdot$~$\cdot$~--}, the adaptive memory method of MacDonald \textit{et al.} \cite{MacDonald2015}.}
	\label{fig:figure11}
\end{figure}
Fig.~\ref{fig:figure11}(a) shows the elongation length predicted by each method as a function of time with comparison with the analytic solution. 
The full memory and two adaptive memory methods are found to well predict the creep response of a fractional Kelvin-Voigt model, while the result predicted by the fixed memory method deviates from the analytic solution after the memory length. 
Fig.~\ref{fig:figure11}(b) shows L1-norm errors of numerical methods as a function of time. 
Obviously, the full memory method shows the most accurate result with a smooth decaying curve along time, although the full memory method requires lots of memory and computational cost. 
The error of the fixed memory method is found to be drastically increased since earlier time points start to be truncated after the memory length. Meanwhile, the error of the adaptive memory method of MacDonald {\it et al.} starts deviating away from the error of the full memory method after $t=2T$ because the convolution weight is inaccurately approximated to account for eliminated points, which degenerates the accuracy. 
However, the present adaptive method shows a notably reduced error which is also very close to that of the full memory method. 
The error of the present method gradually decreases with little bumps when the maximum length between time points increases at $t=2^l {T}, l\in\mathbb{N}$. 

\section{Summary}
\label{sec:sec7}
In the present study, a cost effective and accurate new numerical method for the Caputo fractional derivative has been developed. The present method is based on a novel adaptive memory treatment \cite{MacDonald2015} with the L1 scheme. The present numerical method significantly reduces the amount of computational memory usage and operation counts while achieves high accuracy. Unlike the full memory method, the present method stores time points on a power-law distribution and eliminates unnecessary time points to reduce memory usage. The present method achieves better accuracy by obtaining accurate convolution weights for non-uniformly distributed time points, which is not allowed in the previous adaptive memory method.

In order to analyze the accuracy of the present adaptive memory method, the L1-norm error has been analytically evaluated. The order of accuracy for the present method is $2-\alpha$ order in terms of length between time points. Thus, the total error increases to $2-\alpha$ order along the increments of the maximum length between time points. Also, error behaviors in terms of $\Delta t$ are analytically and numerically investigated. Interestingly, the order of accuracy is found to change in time from $\mathcal{O}(\Delta t^{2-\alpha})$ to $\mathcal{O}(\Delta t^{2})$.

As practical problems, the sub-diffusion process of a time fractional diffusion equation and the creep response of a fractional Kelvin-Voigt model have been simulated and compared with the results of other methods. The present method proves successful reduction of memory requirement and computational cost while providing numerical accuracy much better than the fixed and other adaptive memory methods.

\appendix
\section{$B(m,\alpha) \approx c(\alpha)m^{-\alpha}$}
\label{app:app1}
The function $B(m,\alpha)$ in Eq.~(\ref{eqn:eq34}) can be approximated as $c(\alpha)m^{-\alpha}$ where $c(\alpha)$ is a proportional function for $0<\alpha<1$.
\begin{equation}
\begin{split}
{ B(m,\alpha)=\sum_{k=0}^{m-1}(2m-k)^{1-\alpha}\{2(2m-k-1)+\alpha\}} \\
{-(2m-k-1)^{1-\alpha}\{2(2m-k)-\alpha\}}
\end{split}
\label{app1:eq1}
\end{equation}
By letting $p=2m-k$, Eq.~(\ref{app1:eq1}) is rewritten as follows:
\begin{equation}
{B(m,\alpha)=\sum_{p=m+1}^{2m}p^{1-\alpha}\{2(p-1)+\alpha\}-(p-1)^{1-\alpha}(2p-\alpha)}.
\label{app1:eq2}
\end{equation}
Using a Taylor series expansion, $(p-1)^{1-\alpha}$ is expanded as follows:
\begin{equation}
\begin{split} 
(p-1)^{1-\alpha}=& p^{1-\alpha}-(1-\alpha)p^{-\alpha}-\frac{1}{2}(1-\alpha)\alpha p^{-\alpha-1}-\frac{1}{6}(1-\alpha)\alpha(\alpha+1)p^{-\alpha-2}\\
&-\frac{1}{24}(1-\alpha)\alpha(\alpha+1)(\alpha+2)p^{-\alpha-3}-\mathcal{O}(p^{-\alpha-4}).
\end{split}
\label{app1:eq3}
\end{equation}
By substituting Eq.~(\ref{app1:eq3}) into Eq.~(\ref{app1:eq2}), Eq.~(\ref{app1:eq2}) is recast in terms of $p$ as follows:
\begin{equation}
{B(m,\alpha)=\frac{1}{6}\alpha(1-\alpha)(2-\alpha)\sum_{p=m+1}^{2m}p^{-\alpha-1}+\mathcal{O}(p^{-\alpha-2})}.
\label{app1:eq4}
\end{equation}
The summation in Eq.~(\ref{app1:eq4}) has a following inequality:
\begin{equation}
\begin{split}
&{\sum_{p=m+1}^{2m}(2m)^{-\alpha-1}<\sum_{p=m+1}^{2m} p^{-\alpha-1}<\sum_{p=m+1}^{2m}m^{-\alpha-1}}, \\
&{(2^{-\alpha-1})m^{-\alpha}<\sum_{p=m+1}^{2m} p^{-\alpha-1}<m^{-\alpha}}.   
\end{split}
\label{app1:eq5}   
\end{equation}
Therefore, the function $B(m,\alpha)$ is approximated with a proportional function $c(\alpha)$ as follows:
\begin{equation}
{B(m,\alpha) \approx c(\alpha)m^{-\alpha}}.
\label{app1:eq6}
\end{equation}

\section*{Acknowledgments}
This research was supported by the Basic Science Research Program of the National Research Foundation of Korea (NRF) funded by the Ministry of Science, ICT and Future Planning (NRF-2015R1A2A1A15056086 and NRF-2014R1A2A1A11049599).

\bibliographystyle{siamplain}
\bibliography{references}

\end{document}